\documentclass[envcountsect, envcountsame]{svmult}
\usepackage{etex}

\usepackage{mathptmx}       
\usepackage{helvet}         
\usepackage{courier}        
\usepackage{type1cm}        
\usepackage{makeidx}         
\usepackage{graphicx}        
\usepackage{multicol}        
\usepackage[bottom]{footmisc}

\makeindex             

\usepackage{amstext,amsfonts,amssymb,amscd,amsbsy,amsmath}
\usepackage{tikz-cd}	
\usepackage{tikz}
\usetikzlibrary{positioning, arrows, decorations.markings}
\usepackage[all]{xy}
\usepackage{url}

\newcommand{\PP}{\mathbb{P}}

\newcommand{\RR}{\mathbb{R}}
\newcommand{\ZZ}{\mathbb{Z}}

\DeclareMathOperator{\trop}{trop}
\DeclareMathOperator{\Trop}{Trop}

\DeclareMathOperator{\tr}{tr}
\DeclareMathOperator{\Del}{Del}
\DeclareMathOperator{\Stab}{Stab}
\DeclareMathOperator{\cogr}{cogr}

\newcommand{\tropcurve}{weighted metric graph }
\newcommand{\tropcurves}{weighted metric graphs }

\smartqed

\begin{document}

\title*{From Curves to Tropical Jacobians and Back}
\author{Barbara Bolognese, Madeline Brandt, Lynn Chua}
\institute{Barbara Bolognese \at The University of Sheffield, Sheffield, UK \email{ b.bolognese@sheffield.ac.uk}
\and Madeline Brandt \at Department of Mathematics, University of California, Berkeley, 970 Evans Hall, Berkeley, CA, 94720, \email{brandtm@berkeley.edu}
\and Lynn Chua \at Department of Electrical Engineering and Computer
Science, University of California, Berkeley, 643 Soda Hall, Berkeley,
CA, 94720, \email{chualynn@berkeley.edu}}
%
%
\maketitle

\abstract*{Given a curve defined over an algebraically closed field which is complete with respect to a nontrivial valuation, we study its tropical Jacobian. This is done by first tropicalizing the curve, and then computing the Jacobian of the resulting weighted metric graph. In general, it is not known how to find the abstract tropicalization of a curve defined by polynomial equations, since an embedded tropicalization may not be faithful, and there is no known effective algorithm for carrying out semistable reduction in practice. We solve this problem in the case of hyperelliptic curves by studying admissible covers. We also describe how to take a weighted metric graph and compute its period matrix, which gives its tropical Jacobian and tropical theta divisor. Lastly, we describe the present status of reversing this process, namely how to compute a curve which has a given matrix as its period matrix.}

\abstract{
Given a curve defined over an algebraically closed field which is complete with respect to a nontrivial valuation, we study its tropical Jacobian. This is done by first tropicalizing the curve, and then computing the Jacobian of the resulting weighted metric graph. In general, it is not known how to find the abstract tropicalization of a curve defined by polynomial equations, since an embedded tropicalization may not be faithful, and there is no known algorithm for carrying out semistable reduction in practice. We solve this problem in the case of hyperelliptic curves by studying admissible covers. We also describe how to take a weighted metric graph and compute its period matrix, which gives its tropical Jacobian and tropical theta divisor. Lastly, we describe the present status of reversing this process, namely how to compute a curve which has a given matrix as its period matrix.}

\section{Introduction}
We describe the process of taking a curve and finding its tropical Jacobian. We aim to carry out each step as algorithmically as possible, however, some steps cannot yet be completed in such a way. We now give a brief overview of the steps involved, which we depict in Figure \ref{structure}.

 Let $\Bbbk$ be an algebraically closed field which is complete with respect to a non-archimedean valuation $v$, and let $X$ be a nonsingular curve of genus $g$ over $\Bbbk$. Let $R$ be the valuation ring of $\Bbbk$ with maximal ideal $m$, and let $k = R/m$ be its residue field. We can associate to $X$ its \emph{abstract tropicalization}, which is the dual \tropcurve $\Gamma$ of the special fiber of a semistable model of $X$. 
In practice, finding the abstract tropicalization of a general curve
is difficult and there is no known algorithm to do this in general
\cite[Remark~3]{CJ}. In this paper, we solve this problem for hyperelliptic curves, which is new to the literature, and discuss known results towards finding abstract tropicalizations of all curves.

 Given $\Gamma$, we compute its \emph{period matrix} $Q_\Gamma$. This corresponds to the \emph{tropical Jacobian} of the curve $X$. By taking the Voronoi decomposition dual to the Delaunay subdivision corresponding to $Q_\Gamma$, we obtain the \emph{tropical theta divisor}. 
This process can also be inverted. The set of period matrices that arise as the tropical Jacobian of a curve is the \emph{tropical Schottky locus}. Starting with a principally polarized tropical abelian variety whose period matrix $Q$ is known to lie in the tropical Schottky locus, we give a procedure to compute a curve whose tropical Jacobian corresponds to $Q$.

This process of associating a tropical Jacobian to a curve can also be carried out by looking at classical Jacobians of curves. Jacobians of curves are principally polarized abelian varieties in a natural way; they are the most well known and extensively studied among abelian varieties. Both algebraic curves and abelian varieties have extremely rich geometries, which cannot be fully understood in many cases. Jacobians  provide a link between such geometries, and they often reveal hidden features of algebraic curves which cannot be uncovered otherwise. In order to associate a tropical Jacobian to a complex algebraic curve $C$, one first constructs its classical Jacobian
\begin{equation}
J(C) :  = H^0(C, \omega _C)^*/H_1(C, \ZZ ),
\end{equation}
where we denote by $\omega _C$ the cotangent bundle of the curve. This complex torus admits a natural principal polarization $\Theta$, called the {\it theta divisor}, such that the pair $(J(C), \Theta)$ is a principally polarized abelian variety. We can then obtain the tropical Jacobian by taking the Berkovich skeleton of the classical Jacobian. Baker-Rabinoff, and independently Viviani, proved that this alternative path gives the same result as the procedure we describe in this paper \cite{jacobians, Viviani}.  However, the classical process hides more difficulties on the computational level and proves much more challenging to carry out in explicit examples. Methods have been implemented in the \emph{Maple} package \emph{algcurves} for computing Jacobians numerically over $\mathbb{C}$ \cite{algcurves}.
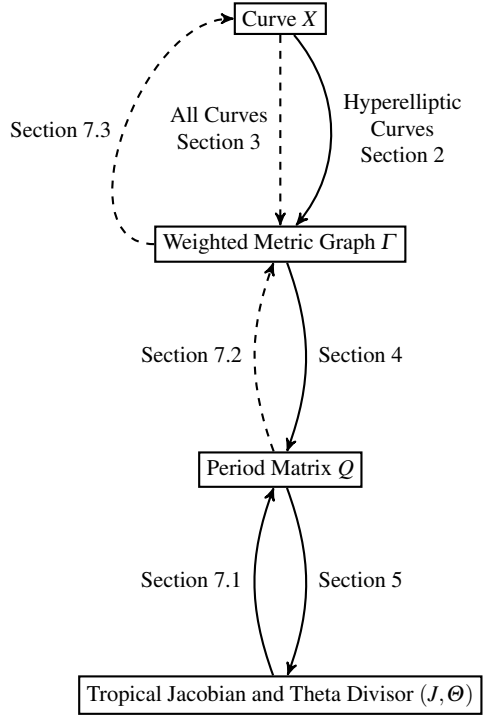
\begin{figure}
\begin{center}
\begin{tikzpicture}[->,>=stealth',auto,node distance=3cm,
  thick,main node/.style={draw}]

  \node[main node] (1) {Curve $X$};
  \node[main node] (2) [below of=1] {Weighted Metric Graph $\Gamma$};
  \node[main node] (3) [below of=2] {Period Matrix $Q$};
  \node[main node] (4) [below of=3] {Tropical Jacobian and Theta Divisor $(J,\Theta)$};

\draw [->](1)  to [out=310, in=50]  node [midway,right]
 {\begin{tabular}{c} Hyperelliptic\\ Curves\\ Section 2 \end{tabular}} 
 (2) ;
\draw [dashed,->](1)  to [out=270, in=90]  node [midway,left] 
 {\begin{tabular}{c} All Curves\\ Section 3 \end{tabular}}  
(2) ;
\draw [->](2)  to [out=290, in=70]  node [midway,right]
 {\begin{tabular}{c} Section 4 \end{tabular}}  
 (3);
\draw [->](3)  to [out=290, in=70]  node [midway,right]
 {\begin{tabular}{c} Section 5 \end{tabular}}  
 (4);
\draw [->](4)  to [out=110, in=250]  node [midway,left]
 {\begin{tabular}{c}Section 7.1 \end{tabular}}  
 (3);
 \draw [dashed,->](3)  to [out=110, in=250]  node [midway,left]
 {\begin{tabular}{c}Section 7.2 \end{tabular}}  
 (2);
  \draw [dashed,->](2)  to [out=180, in=180]  node [midway,left]
 {\begin{tabular}{c}Section 7.3 \end{tabular}}  
 (1);

\end{tikzpicture}
\caption{The structure of the paper, and the process of associating a tropical Jacobian to a curve. Dashed arrows indicate steps which are not yet achievable algorithmically.}
\label{structure}
\end{center}
\end{figure}

The structure of this paper is depicted in Figure \ref{structure}. In Section \ref{hyperelliptic} we find the abstract tropicalization of all hyperelliptic curves, a result which is new to this paper. Then, in Section \ref{generalcurves} we discuss issues with embedded tropicalization, state some known results about certifying a faithful tropicalization, and outline the process of semistable reduction. This step of the procedure is far from being algorithmic, and this section focuses on examples which provide obstacles to doing this in general. In Section \ref{periodmatrix} we describe how to find the period matrix of a weighted metric graph. Then we define and give examples of the tropical Jacobian and its theta divisor in Section \ref{jacobian}. In Section \ref{section-schottky}, we discuss the tropical Schottky problem. Finally, we describe the obstacles to reversing this process in Section \ref{andback}.

\section{ Hyperelliptic Curves}
\label{hyperelliptic}

We now study the problem of finding the abstract tropicalization of hyperelliptic curves.
 In the case of elliptic curves, the tropicalization can be completely described in terms of the $j$-invariant \cite{KMM}. 
 Similarly, tropicalizations of genus 2 curves (all of which are hyperelliptic) can be described by studying \emph{tropical Igusa invariants} \cite{helminck}.
 This problem was also solved in genus 2 by studying the curve as a double cover of $\mathbb{P}^1$ ramified at 6 points, as done in \cite[Section 5]{RSS}.

In this section, we generalize the latter method to find tropicalizations of all hyperelliptic curves, which was previously not known. 
Let $X$ be a nonsingular hyperelliptic curve of genus $g$ over $\Bbbk$, an algebraically closed field which is complete with respect to a nontrivial, non-archimedean valuation $v$. Our goal is to find $\Gamma$, the abstract tropicalization of $X$. 

We denote by $\mathcal{M}_{g,n}$ the moduli space of genus $g$ curves with $n$ marked points, see \cite{harris} for a thorough introduction. The space $\mathcal{M}_{0,2g+2}$ maps surjectively onto the hyperelliptic locus inside $\mathcal{M}_g$ by identifying each hyperelliptic curve of genus g with a double cover of $\mathbb{P}^1$ ramified at $2g+2$ marked points. When the characteristic of $\Bbbk$ is not 2, the normal form for the equation of a hyperelliptic curve is
$
y^2 = f(x),
$
where $f(x)$ has degree $2g+2$, and the roots of $f$ are distinct. Then, these are precisely the ramification points.

The space $\mathcal{M}_{0,2g+2}^{tr}$ is the tropicalization of $\mathcal{M}_{0,2g+2}$. A \emph{phylogenetic tree} is a metric tree with leaves labeled $\{1, \ldots, m\}$ and no vertices of degree 2. Such a tree is uniquely specified by the distances $d_{ij}$ between the leaves.
We see that $\mathcal{M}_{0,2g+2}^{tr}$ parametrizes the space of phylogenetic trees with $2g+2$ leaves using the Pl\"ucker embedding to map $\mathcal{M}_{0,2g+2}$ into the Grassmannian $Gr(2,\Bbbk^{2g+2})$ \cite[Chapter 4.3]{tropicalbook} :
\begin{equation}
\{(a_i: b_i)\}_{i=1}^{2g+2} \mapsto (p_{12}: p_{13}: \cdots: p_{2g+1, 2g+2}) \ \ \text{where}\ \ p_{i,j} = a_ib_j - a_jb_i.
\end{equation}
The distances are then given by $d_{ij} = -2p_{ij} + n\cdot \bf{1}$ for a suitable constant $n$ \cite[Section 5]{RSS}.
  Using the Neighbor Joining Algorithm \cite[Algorithm 2.41]{PS05}, one can construct the unique tree, along with the lengths of its interior edges, using only the leaf distances $d_{ij}$ as input. 
  Since the lengths of leaf edges can only be defined up to adding a constant length to each leaf, we think of this tree as a metric graph where the leaves have infinite length and the interior edges have lengths as described.
 
This realizes $\mathcal{M}_{0,2g+2}^{tr}$ as a $2g-1$ dimensional fan inside $\mathbb{TP}^{{2g+2 \choose 2}-1}$ (cf. \cite[Section 2.5]{tropicalbook}). The space $\mathcal{M}_{0,2g+2}^{tr}$ can be computed as a tropical subvariety of $\mathbb{TP}^{{2g+2 \choose 2}-1}$, since it has a tropical basis given by the Pl\"{u}cker relations for $Gr(2,\Bbbk^{2g+2})$ \cite[Chapter 4.4]{tropicalbook}. Each cone corresponds to a combinatorial type of tree (see Figure \ref{genus3poset}), and the dimension of each cone corresponds to the number of interior edges in the tree.
 
The next step in finding the tropicalization of the hyperelliptic curve $X$ is to take the corresponding point in $\mathcal{M}_{0,2g+2}^{tr}$, as a tree on $2g+2$ leaves, and compute a \tropcurve in $\mathcal{M}_{g}^{tr}$. Figure \ref{genus3poset} gives this correspondence in the case $g=3$. We now give some definitions related to metric graphs in order to describe this correspondence for general $g$, following \cite{melody}.
 
  \begin{figure}
  \centering
  \includegraphics[width=1.0\linewidth, height = 5 in]{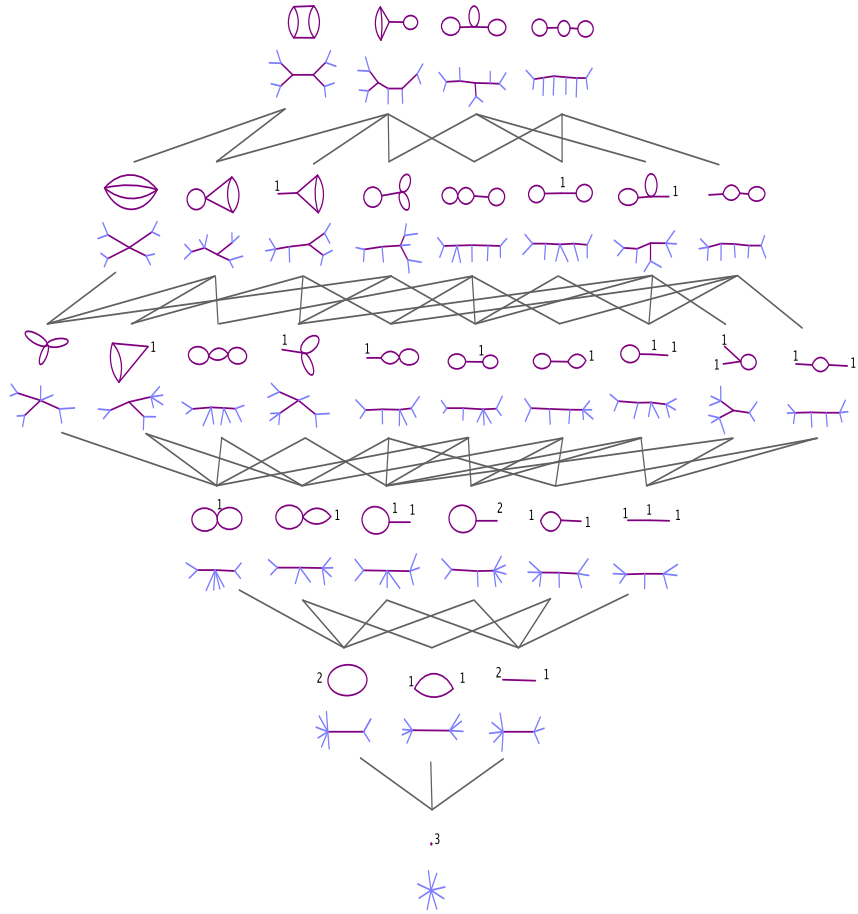}
  \caption{The poset of unlabeled trees with 8 leaves, and tropicalizations of hyperelliptic curves of genus 3. Both are ordered by the relation of contracting an edge.}
  \label{genus3poset}
  \end{figure}

\begin{definition}
A \emph{metric graph} is a metric space $\Gamma$, together with a graph $G$ and a length function $l: E(G) \rightarrow \mathbb{R}_{>0}\cup\{\infty\}$ such that $\Gamma$ is obtained by gluing intervals $e$ of length $l(e)$, or by gluing rays to their endpoints, according to how they are connected in $G$. In this case, the pair $(G,l)$ is called a \emph{model} for $\Gamma$.
A \emph{\tropcurve} is a metric graph $\Gamma$ together with a weight function on its points $w:\Gamma \rightarrow \mathbb{Z}_{\geq 0}$, such that $\sum_{v \in \Gamma}w(v)$ is finite. 
\end{definition}

 We call edges of infinite length \emph{infinite leaves}, and these only meet the rest of the graph in one endpoint. A \emph{bridge} is an edge whose deletion increases the number of connected components. 

 The \emph{genus} of a \tropcurve ($\Gamma$, $w$) is 
\begin{equation}
\label{genusgamma}
 \sum_{v \in \Gamma} w(v) + |E(G)|-|V(G)|+1,
 \end{equation}
 where $G$ is any model of $\Gamma$. 
 We say that two \tropcurves of genus $\geq 2$ are isomorphic if one can be obtained from the other via graph automorphisms, or by removing infinite leaves or leaf vertices $v$ with $w(v) = 0$, together with the edge connected to it. In this way, every \tropcurve has a \emph{minimal skeleton}. 

A model is \emph{loopless} if there is no vertex with a loop edge. The \emph{canonical loopless model} of $\Gamma$, with genus of $\Gamma \geq 2$ , is the graph $G$ with vertices
\begin{align}
V(G) := \{x \in \Gamma\ |\ \text{val}(x)\not = 2\ \text{or}\ w(x)>0\ \text{or}\ x\ \text{is the midpoint of a loop}\}\,.
\end{align}
If $(G,l)$ and $(G',l')$ are loopless models for metric graphs $\Gamma$ and $\Gamma'$, then a \emph{morphism of loopless models} $\phi: (G,l) \rightarrow (G',l')$ is a map of sets ${V(G) \cup E(G) \rightarrow V(G') \cup E(G')}$ such that
\begin{itemize}
\item All vertices of $G$ map to vertices of $G'$.
\item If $e\in E(G)$ maps to $v\in V(G')$, then the endpoints of $e$ must also map to $v$.
\item If $e\in E(G)$ maps to $e'\in E(G')$, then the endpoints of $e$ must map to vertices of $e'$.
\item Infinite leaves in $G$ map to infinite leaves in $G'$.
\item If $\phi(e) = e'$, then $l'(e')/l(e)$ is an integer. These integers must be specified if the edges are infinite leaves.
\end{itemize}
  We call an edge $e \in E(G)$ \emph{vertical} if $\phi$ maps $e$ to a vertex of $G'$. We say that $\phi$ is \emph{harmonic} if for every $v \in V(G)$, the \emph{local degree}
\begin{equation}
d_v = \sum_{\substack{
e\ni v,\\ \phi(e) = e'
}} l'(e') / l(e)
\end{equation}
is the same for all choices of $e' \in E(G')$. If it is positive, then $\phi$ is \emph{nondegenerate}. The \emph{degree} of a harmonic morphism is defined as 
\begin{equation}
\sum_{\substack{e \in E(G),\\ \phi(e) = e'}} l'(e')/l(e).
\end{equation}
We also say that $\phi$ satisfies the \emph{local Riemann-Hurwitz condition} if:
\begin{equation}
2 - 2w(v) = d_v(2 - 2w'(\phi(v))) - \sum_{e\ni v}\left(l'(\phi(e))/l(e) - 1\right).
\end{equation}
If $\phi$ satisfies this condition at every vertex $v$ in the canonical loopless model of $\Gamma$, then $\phi$ is called an \emph{admissible cover} \cite{admissiblecover}.

\begin{definition}\cite[Theorem 1.3]{melody}
Let $\Gamma$ be a weighted metric graph, and let $(G, l)$ denote its canonical loopless model. We say that $\Gamma$ is \emph{hyperelliptic} if there exists a nondegenerate harmonic morphism of degree 2 from $G$ to a tree.
\end{definition}
 A hyperelliptic curve will always tropicalize to a hyperelliptic weighted metric graph, however not every hyperelliptic \tropcurve is the tropicalization of a hyperelliptic curve.

\begin{theorem}\cite[Corollary 4.15]{liftingharmonic} Let $\Gamma$ be a minimal \tropcurve of genus $g \geq 2$. Then there is a smooth proper hyperelliptic curve $X$ over $\Bbbk$ of genus $g$ having $\Gamma$ as its minimal skeleton if and only if $\Gamma$ is hyperelliptic and for every $p \in \Gamma$ the number of bridge edges adjacent to $p$ is at most $2 w(p)+2$.
\end{theorem}

Lemma \ref{treelemma} and its proof give an algorithm for taking a tree with $2g+2$ infinite leaves and obtaining a metric graph which is an admissible cover of the tree.

\begin{lemma}
\label{treelemma}
 Every tree $T$ with $2g+2$ infinite leaves has an admissible cover $\phi$ by a unique hyperelliptic metric graph $\Gamma$ of genus $g$, and $\phi$ is harmonic of degree 2.
\end{lemma}
\begin{proof} 
Let $T$ be a tree with $2 g + 2$ infinite leaves. If all infinite leaves are deleted, then a finite tree $T'$ remains. Let $v_1, \ldots, v_k$ be the vertices of $T'$, ordered such that the distance from $v_i$ to $v_k$ is greater than or equal to the distance from $v_j$ to $v_k$, for $i<j$.

We construct $\Gamma$ iteratively by building the preimage of each vertex $v_i$, asserting along the way that the local Riemann-Hurwitz condition holds. This also gives an algorithm for finding $\Gamma$. We begin with $v_1$, which has a positive number $n_1$ of leaf edges in $T$. 
Since $\phi$ has degree 2, it must be locally of degree 1 or 2 at every vertex of $\Gamma$. Since the preimage of each infinite leaf must be an infinite leaf, we will attach $n_1$ infinite leaves at the preimage $\phi^{-1}(v_1)$ in $\Gamma$.

At any vertex in $\Gamma$ with infinite leaves, $\phi$ has local degree $2$, hence we will attach to $\Gamma$ an infinite leaf $e$ such that $l(\phi(e))/l(e) = 2$. Then, there is a unique vertex in the preimage $\phi^{-1}(v_1)$. Otherwise, there would need to be another edge in the preimage of each leaf, so the degree of the morphism would be greater than 2.

Let $e_1$ be the edge connecting $v_1$ to some other $v_i$. There are two possibilities:
\begin{enumerate}
\item The preimage of $e_1$ is two edges in $\Gamma$, each with length $l(e_1)$. The local Riemann-Hurwitz equation reads
\begin{equation}
2-2w(\phi^{-1}(v_1)) = 2(2-0)-(n_1 + 0+ 0).
\end{equation}
This is only possible if $n_1$ is even, and $\phi^{-1}(v_1)$ has weight $(n_1-2)/2$.
\item The preimage of $e_1$ is one edge in $\Gamma$, with length $l(e_1)/2$. The local Riemann-Hurwitz equation reads
\begin{equation}
2-2w(\phi^{-1}(v_1)) = 2(2-0)-(n_1 + 1).
\end{equation}
This is only possible if $n_1$ is odd, and $\phi^{-1}(v_1)$ has weight $(n_1-1)/2$.
\end{enumerate}
Now, we proceed to the other vertices. As long as the order of the vertices is respected, at each vertex $v_i$ there will be at most one edge $e_i$ whose preimage in $\Gamma$ we do not know. Then, what happens at $v_i$ can be completely determined by studying the local Riemann-Hurwitz data. For $i>1$, let $n_i$ be the number of infinite leaves at $v_i$ plus the number of edges $e \in T$ such that $e = \{v_i, v_j\}$, $j<i$, and $\phi^{-1}(e)$ is a bridge in $\Gamma$. If $n_i>0$, then either 1 or 2 holds. However, it is possible that $n_i = 0$, in which case we have a third possibility:
\begin{enumerate}
\setcounter{enumi}{2}
\item If $n_i = 0$, let $v'_i \in \phi^{-1}(v_i)$. The local Riemann-Hurwitz equation reads:
\begin{equation}
2-2w(v_i') = d_{v_i}(2-0)-(0).
\end{equation}
Then we must have $d_{v_i} = 1$ and $w(v_i')=0$, which implies that there are two vertices in $\phi^{-1}(v_i)$.
\end{enumerate}
Finally, we glue the pieces of $\Gamma$ as specified by $T$, and contract the leaf edges on $\Gamma$.
The fact that $\Gamma$ has genus $g$ is a consequence of the local Riemann-Hurwitz condition. 
\qed \end{proof}

We remark that this process did not require the fact that the tree had an odd number of leaves. Indeed, if one repeats this procedure for such a tree, a hyperelliptic metric graph will be obtained. However, this graph is not the tropicalization of a hyperelliptic curve.

\begin{example}
\label{algoex}
In Figure \ref{fig-hyperelliptic-ex}, we have a tree with vertices labelled $v_1,\ldots, v_7$. Beginning with $v_1$, we observe that $n_1 = 2$, which means that the edge from $v_1$ to $v_3$ has two edges in its preimage. The same is true for $v_2$. Moving on to $v_3$, we see that $n_3=0$, which means that $v_3$ has two points in $\Gamma$ which map to it. We can connect the edges from $\phi^{-1}(v_1)$ and $\phi^{-1}(v_2)$ to the two points in $\phi^{-1}(v_3)$. Since $\phi^{-1}(v_3)$ has two points, the edge from $v_3$ to $v_4$ corresponds to two edges in $\Gamma$, so $n_4 = 2$, which means that the edge from $v_4$ to $v_5$ also splits. Next, $n_5 = 1$, which means that the edge from $v_5$ to $v_6$ corresponds to a bridge in $\Gamma$. Then, $n_6 = 4$, which means that the edge $v_6$ to $v_7$ splits, and the vertex mapping to $v_6$ has genus $1$. Lastly, since $n_7=2$, the point mapping to $v_7$ has genus 0. All edges depicted in the image have the same length as the corresponding edges in the tree, except for the bridge, which has length equal to half the length of the corresponding edge in the tree.

\begin{figure} 
\begin{center}
\includegraphics[height=2.3in]{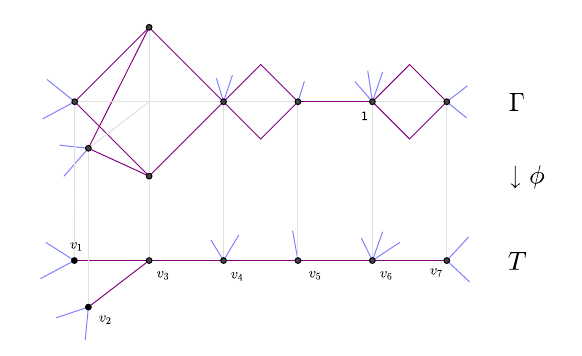}
\caption{The tree $T$ with 12 infinite leaves from Example \ref{algoex} and the hyperelliptic \tropcurve $\Gamma$ of genus $5$ which admissibly covers $T$ by $\phi$.}
\label{fig-hyperelliptic-ex}
\end{center}
\end{figure}

\end{example}

The following theorem shows that this metric graph is actually the tropicalization of a hyperelliptic curve.

\noindent
\begin{minipage}{\textwidth}
\begin{theorem} Let $g \geq 1$ be an integer. Let $X$ be a hyperelliptic curve of genus $g$ over $\Bbbk$, given by taking the double cover of $\mathbb{P}^1$ ramified at $2g+2$ points $p_1, \ldots, p_{2g+2}$. If $T$ is the tree which corresponds to the tropicalization of $\mathbb{P}^1$ with the marked points $p_1, \ldots, p_{2g+2}$ described above, and $\Gamma$ is the unique hyperelliptic \tropcurve which admits an admissible cover to $T$, then $\Gamma$ is the abstract tropicalization of $X$.
\end{theorem}
\end{minipage}

\begin{proof}
This follows from \cite{admissiblecover}, Remark 20 and Theorem 4. Indeed, the hyperelliptic locus of $\mathcal{M}_g$ can be understood as the space  $\overline{H}_{g\rightarrow 0,2}((2),\ldots, (2))$ of admissible covers with $2g+2$ ramification points of order 2. Its tropicalization is constructed and studied in \cite{admissiblecover}. The space $\overline{H}^{an}_{g\rightarrow 0,2}((2),\ldots, (2))$ is the Berkovich analytification of $\overline{H}_{g\rightarrow 0,2}((2),\ldots, (2))$, and thus a point $X$ is represented by an admissible cover over $\text{Spec}(K)$ with $2g+2$ ramification points of order 2. By Theorem 4 in \cite{admissiblecover}, the diagram
\begin{equation} 
\xymatrix{
&\overline{H}^{an}_{g\rightarrow 0,2}((2),\ldots, (2)) \ar[rd]^{src^{an}}\ar[ld]_{br^{an}} \ar[dd]_{\trop} &\\
\overline{\mathcal{M}}^{an}_{0,2g+2} \ar[dd]_{\trop} & & \overline{\mathcal{M}}^{an}_g \ar[dd]_{\trop}  \\
& \overline{H}^{\trop}_{g\rightarrow 0,2}((2),\ldots, (2))  \ar[rd]^{src^{\trop}}\ar[ld]_{br^{\trop}}&\\
\overline{\mathcal{M}}^{\trop}_{0,2g+2} & & \overline{\mathcal{M}}^{\trop}_g  \\
}
\end{equation}
commutes. The morphisms $src$ take a cover to its source curve, marked at the entire inverse image of the branch locus, and the morphisms $br$ take a cover to its base curve, marked at its branch points. We start with an element $X$ of $\overline{H}^{an}_{g\rightarrow 0,2}((2),\ldots, (2))$, and we wish to find $\trop(src^{an}(X)) \in \overline{\mathcal{M}_g}^{trop}$. The unicity in Lemma \ref{treelemma} enables us to find an inverse for $br^{\trop}$. Then $T = \trop(br^{an}(X))$, and so by commutativity of the diagram, $\trop(src^{an}(X)) =  src^{\trop}((br^{\trop})^{-1}(T)) = \Gamma$. 
\qed \end{proof}

\begin{example}(\cite[Problem~2 on Curves]{Sturmfels}) Consider the curve
$$
y^2 = (x-1)(x-2)(x-3)(x-6)(x-7)(x-8)
$$
with the 5-adic valuation. In $\mathcal{M}_{0,6}^{trop}$, this gives us the point
$$
(0,0,1,0,0,0,0,1,0,0,0,1,0,0,0) = (p_{1,2}, p_{1,3}, \ldots, p_{5,6}).
$$
When $n$ is any integer, this gives us a tree metric of
$$
(n,n,n-2,n,n,n,n,n-2,n,n,n,n-2,n,n,n) = (d_{1,2}, d_{1,3}, \ldots, d_{5,6}).
$$
Then, this is a metric for the tree on the left of Figure \ref{bookexample}, and the metric graph on the right.
\begin{figure}
\begin{center}
\includegraphics[height=1.5in]{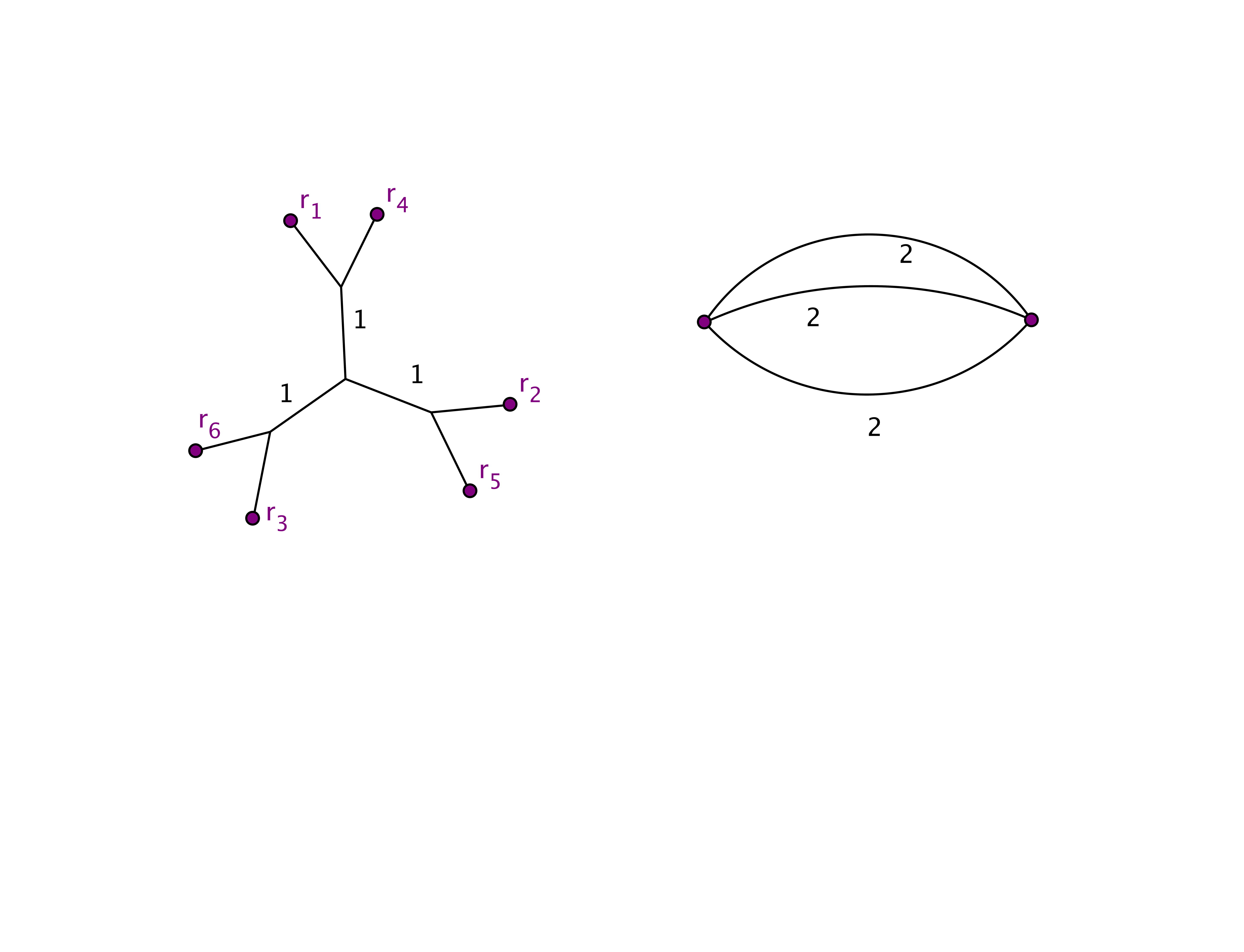}
\caption{This is the tree and metric graph for Problem 2 from the \emph{Curves} worksheet.}
\label{bookexample}
\end{center}
\end{figure}

\end{example}

\section{Other Curves}
\label{generalcurves}

Outside of the hyperelliptic case, finding the abstract tropicalization of a curve is very hard. In this section, we highlight some of the difficulties and discuss two approaches to this problem: faithful tropicalization and semistable reduction. We offer the following example as motivation for why this is a difficult problem.

\begin{example}(\cite[Problem~9 on Abelian Combinatorics]{Sturmfels})
We begin with a curve in $\mathbb{P}^2$, given by the zero locus of
\begin{align*}
 f(x,y,z)
=&\  41x^4+1530x^3y+3508x^3z+1424x^2y^2+2490x^2yz\\
&-2274x^2z^2+470xy^3 +680xy^2z-930xyz^2+772xz^3\\
&+535y^4-350y^3z-1960y^2z^2-3090yz^3-2047z^4,
\end{align*}
defined over $\mathbb{Q}_2$.
Here, the induced regular subdivision of the Newton polygon will be trivial, since the 2-adic valuation of the coefficients on the $x^4,\ y^4,\ z^4$ terms is each 0. Therefore, we can detect no information about the structure of the abstract tropicalization from this embedded tropicalization. In Example \ref{problem9cont}, we will pick a nice enough change of coordinates to allow us to find a faithful tropicalization.
\end{example}

\subsection{Faithful Tropicalization} 
We briefly discuss the \emph{Berkovich skeleton} of a curve. Let $\Bbbk$ be an algebraically closed field which is complete with respect to a nontrivial, nonarchimedean valuation $v$. Let $X$ be a nonsingular  curve defined over $\Bbbk$. To simplify the definitions in this exposition, we restrict to the case when $X$ is affine, and for a thorough explanation the reader may consult \cite{berkovich}. 
The \emph{Berkovich analytification} $X^{an}$ is a topological space whose ground set consists of all multiplicative seminorms on the coordinate ring $\Bbbk[X]$ that are compatible with the valuation $v$ on $\Bbbk$. It has the coarsest topology such that for every $f \in \Bbbk[X]$, the map sending a seminorm $|\cdot |$ to $|f|$ is continuous. We note that this is different from the metric structure, see \cite[5.3]{pbr}.
 When $X$ is a smooth, proper, geometrically integral curve of genus greater than or equal to 1, $X^{an}$ admits strong deformation retracts onto finite metric graphs called \emph{skeletons} of $X^{an}$ \cite{baker}. There is also a unique \emph{minimal skeleton}.

The minimal skeleton of the Berkovich analytification is the same metric graph as the dual graph of the special fiber of a stable model of a curve $X$. Let $i: X \hookrightarrow \mathbb{A}^n$ be an embedding of $X$, which gives generators $f_1, \ldots, f_m$ of the coordinate ring $\Bbbk[X]$. Denote its embedded tropicalization by $\Trop(X,i)$. Details about how to carry this out (including how to find the metric on $\Trop(X,i)$) can be found in \cite{tropicalbook}. We can relate the Berkovich analytification to embedded tropicalizations in the following way.

\begin{theorem}\cite[Theorem 1.1]{inverselimit}
Let $X$ be an affine variety over $\Bbbk$. Then there is a homeomorphism
\begin{equation}
X^{an} \rightarrow \varprojlim \Trop(X,i).
\end{equation}
\end{theorem}
The homeomorphism is given by the inverse limit of maps $\pi_i: X^{an} \rightarrow \mathbb{A}^m$ defined by
\begin{equation}
\pi_i (x) = (-\log |f_1|_x, \ldots, -\log |f_m|_x),
\end{equation}
where $|\cdot|_x$ denotes the norm corresponding to the point $x \in X^{an}$. The image of this map is equal to $\Trop(X,i)$.

Given just one embedded tropicalization, we want to detect information about the Berkovich skeleton. This is the problem of \emph{certifying faithfulness}, as studied in \cite[5.23]{bpr}.

In some cases, the embedded tropicalization contains enough information to determine the structure of the skeleton of $X^{an}$.

\begin{theorem} Let $X$ be a smooth curve in $\mathbb{P}_\Bbbk^n$ of genus $g$. Further, suppose that $\dim(H_1(\Trop(X),\mathbb{R}))=g$, all vertices of $\Trop(X)\subset \mathbb{R}^{n+1}/\mathbb{R}\mathbf{1}$ are trivalent, and all edges have multiplicity 1. Then the minimal skeleta of $\Trop(X)$ and $X^{an}$ are isometric. In particular, if $X$ is a smooth curve in $\mathbb{P}_\Bbbk^2$ whose Newton polygon and subdivision form a unimodular triangulation, then the minimal skeleta of $\Trop(X)$ and $X^{an}$ are isometric.
\label{bprthm}
\end{theorem}
\begin{proof} 
This follows from a result of Baker, Payne, and Rabinoff \cite[Corollary 5.2.8]{bpr} who assume instead that all vertices of $\Trop(X)\subset \mathbb{R}^{n+1}/\mathbb{R}1$ are trivalent, all edges have multiplicity 1, $\Sigma$ has no leaves, and that $\dim(H_1(\Trop(X),\mathbb{R})) = \dim(\Sigma,\mathbb{R}))$.
 In our case, we are trying to detect information about the minimal skeleton $\Sigma$ of $X^{an}$, so we have reduced to the case in which we can remove assumptions about $\Sigma$ from the hypotheses. 
 
 In particular, we have that ${g \geq \dim (H_1(\Sigma , \mathbb{R})) \geq \dim (H_1(\Trop(X),\mathbb{R}))}$, so in the case when $g= \dim(H_1(\Trop(X),\mathbb{R}))$, we also have $g= \dim(H_1(\Sigma , \mathbb{R}))$. Furthermore, it is only possible that $\Sigma$ has leaves when  $g>\dim(H_1(\Sigma , \mathbb{R}))$.
\qed \end{proof}

The next example illustrates how to apply this Theorem to find the metric graph of the curve given in Problem 9.

\begin{example}[Problem 9, Continued]
\label{problem9cont}
We apply a change of coordinates
$$
x = \frac{1}{12}X+\frac{1}{2}Y-\frac{1}{12}Z,\ \ 
y = \frac{1}{2}X-\frac{1}{2}Y,\ \ 
z = -\frac{5}{12}X-\frac{1}{12}Z.
$$
to obtain
$$
0 = -256X^3Y-2X^2Y^2-256XY^3-8X^2YZ-8XY^2Z-XYZ^2-
2XZ^3-2YZ^3.
$$
We calculate the regular subdivision of the Newton polyton in \emph{Polymake 3.0} \cite{polymake}, weighted by the 2-adic valuations of the coefficients. 

The embedded tropicalization and corresponding metric graph, with edge lengths, are depicted in Figure~\ref{ex9fig}. Since we have that all vertices are trivalent, all edges have multiplicity 1, and $\dim(H_1(\Trop(X),\mathbb{R})) = 3$, we have by Theorem \ref{bprthm} that this is the abstract tropicalization of the curve.
\begin{figure}
\centering
\begin{minipage}{.5\textwidth}
  \centering
  \includegraphics[width=.7\linewidth]{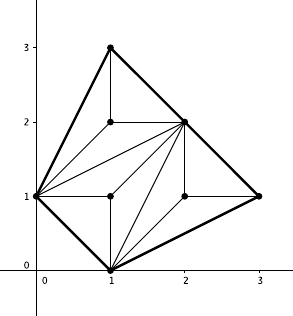}
\end{minipage}%
\begin{minipage}{.5\textwidth}
  \centering
\includegraphics[width=.7\linewidth]{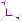}
\end{minipage}
\begin{minipage}{.5\textwidth}
  \centering
\includegraphics[width=.8\linewidth]{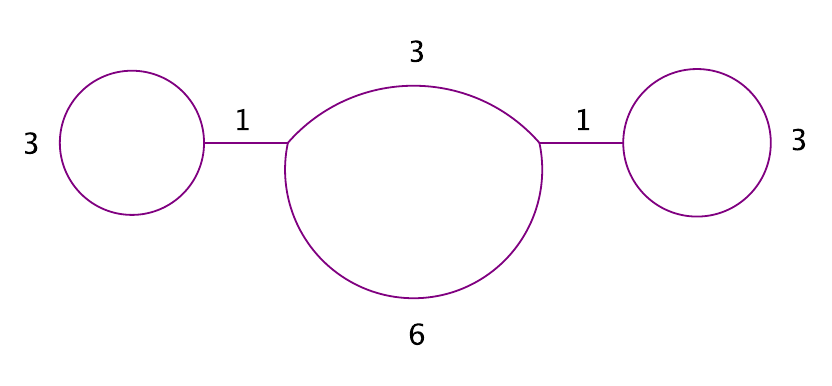}
\end{minipage}
\caption{The Newton polygon from Example \ref{problem9cont} together with its unimodular triangulation on the upper left, the embedded tropicalization on the upper right, and the metric graph at the bottom.}
\label{ex9fig}
\end{figure}
\end{example}

\begin{example} We offer another example to illustrate some of the shortcomings of Theorem \ref{bprthm}. We will ultimately need to use semistable reduction to solve the problem in this case.
\label{melodyexample}
Consider the curve $X$ in $\mathbb{P}^2$ over $\mathbb{C}\{\!\{t\}\!\}$ defined by
$$
x y z^2 + x^2 y^2 + 29t(xz^3+yz^3) + 17t^2(x^3y+xy^3) = 0.
$$
Tropicalizing $X$ with this embedding, we obtain the embedded tropicalization in Figure \ref{melodyexamplefig}.
\begin{figure}
\centering
\begin{minipage}{.5\textwidth}
  \centering
  \includegraphics[width=.7\linewidth]{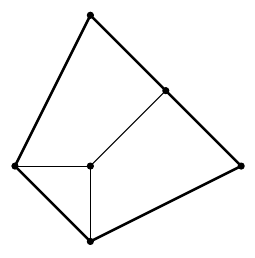}
\end{minipage}%
\begin{minipage}{.5\textwidth}
  \centering
\includegraphics[width=.7\linewidth]{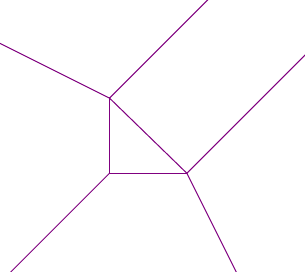}
\end{minipage}
\caption{The Newton polygon from Example \ref{melodyexample} with its regular subdivision on the left, and the embedded tropicalization on the right.}
\label{melodyexamplefig}
\end{figure}
Since this is not a unimodular triangulation, Theorem \ref{bprthm} does not allow us to draw any conclusions. However, we will see in the next section that this is, in fact, a faithful tropicalization.
\end{example}

\subsection{Semistable Reduction}

If we cannot certify faithfulness of a tropicalization (as in Example \ref{melodyexample}), we can instead find the metric graph $\Gamma$ by taking the dual graph of a semistable model for $X$ \cite{pbr}. To this end, we outline the process of finding the semistable model of a curve $X$. 

Let $X$ be a reduced, nodal curve over $\Bbbk$, and for each irreducible component $C$ of $X$, let $\phi: \tilde{C} \rightarrow C$ be the normalization of $C$. We say that $X$ is \emph{semistable} if every smooth rational component meets the rest of the curve in at least two points, or every component of $\tilde{C}$ has at least 2 points $x$ such that $\phi(x)$ is a singularity in $X$.

Let $R$ be the valuation ring of $\Bbbk$. Then $\text{Spec}(R)$ contains two points: one corresponding to the zero ideal $(0)$ and another corresponding to $m$, the maximal ideal of $R$. If $\mathcal{X}$ is a scheme over $\text{Spec}(R)$, we call the fiber of the the point corresponding to $(0)$ the \emph{generic fiber}, and the fiber over the point corresponding to $m$ the \emph{special fiber}.
\begin{definition}
 If $X$ is any finite type scheme over $\Bbbk$, a \emph{model} for $X$ is a flat and finite type scheme $\mathcal{X}$ over $R$ whose generic fiber is isomorphic to $X$. We call this model \emph{semistable} if the special fiber $\mathcal{X}_k = X \times_R k$ is a semistable curve over $k$. 
 \end{definition}
The curve $X$ always admits a semistable model, by the Semistable Reduction Theorem. The proofs of this theorem contain somewhat algorithmic approaches, see \cite{Deligne1969} and \cite{AW}.

Now, we describe a procedure for finding a semistable model for the curve $X$ when $\Bbbk$ has characteristic 0. A good reference is \cite{harris}.
The first step is to blow up the total space $\mathcal{X}$, removing any singularities in the special fiber, to arrive at a family whose special fiber is a nodal curve. At this point, our work is not yet done because the resulting curve will be nonreduced.

\begin{example}(Example \ref{melodyexample}, Continued) \label{ex-semistable-reduction} In this case, the special fiber is a conic with two tangent lines, depicted in Figure \ref{fig-semistable-ab}(a). We denote the conic by $C$ and the two lines by $l_1$ and $l_2$. We begin by blowing up the total space at the point $p_1$. The result, depicted in Figure \ref{fig-semistable-ab}(b), is that $l_1$ and $C$ are no longer tangent, but they do intersect in the exceptional divisor, which we call $e_1$. The exceptional divisor $e_1$ has multiplicity 2, coming from the multiplicity of the point $p_1$. In the figures, we denote the multiplicities of the components with grey integers, and a component with no integer is assumed to have multiplicity 1.
\begin{figure}
\begin{center}
\begin{tabular}{cc}
 \includegraphics[height=1.6in]{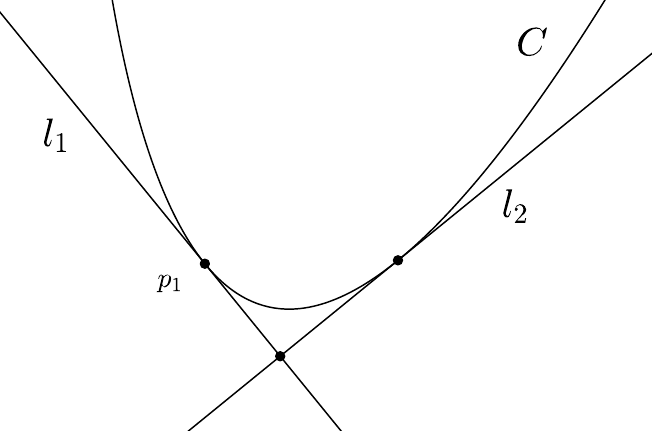} &   \includegraphics[height=1.6in]{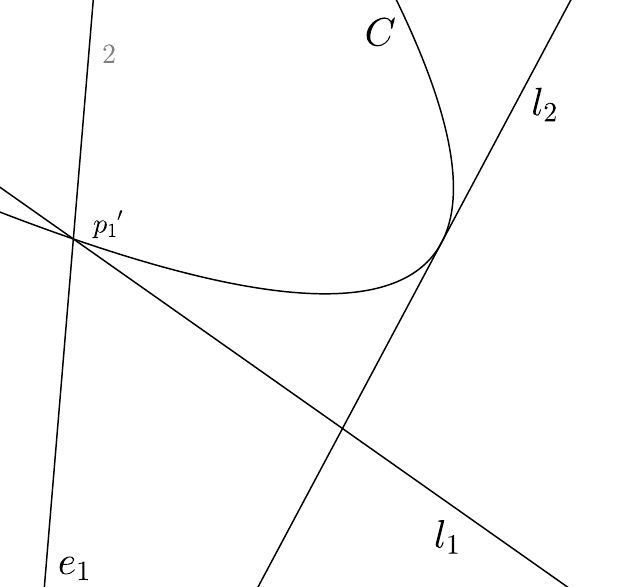} \\
 (a) & (b) \\ \\ 
\end{tabular}
\caption{Semistable reduction in Example \ref{ex-semistable-reduction}.}
\label{fig-semistable-ab}
\end{center}
\end{figure}
Next, we blow up the total space at the point labelled $p_1'$ to get Figure \ref{fig-semistable-cd}(c). We denote the new exceptional divisor by $e_1'$ with multiplicity 4, and the curves $l_1$ and $C$ no longer intersect. All points except $l_2\cap C$ are either smooth or have nodal singularities, so we repeat these two blowups here, obtaining the configuration in Figure \ref{fig-semistable-cd}(d).
\begin{figure}
\begin{center}
\begin{tabular}{cc}
 \includegraphics[height=1.6in]{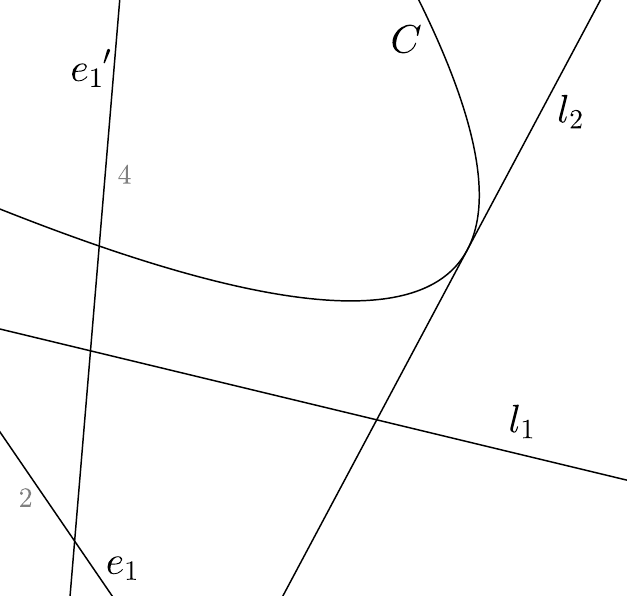} &   \includegraphics[height=1.6in]{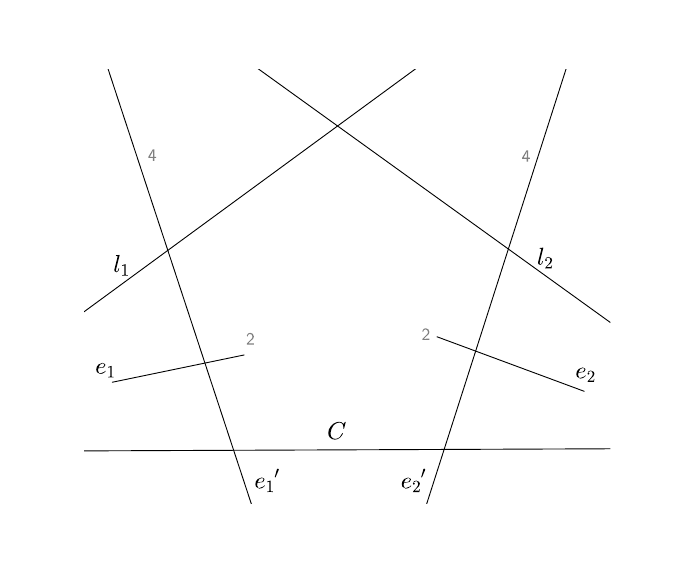} \\
  (c) & (d) \\
\end{tabular}
\caption{Semistable reduction in Example \ref{ex-semistable-reduction}.}
\label{fig-semistable-cd}
\end{center}
\end{figure}
\end{example}

\begin{remark} At this point, we have a family whose special fiber only has nodes as singularities, but it is not reduced. To fix this, we make successive base changes of prime order $p$. Explicitly, we take the $p$-th cover of the family branched along the special fiber. Then, if $D$ is a component of multiplicity $q$ in the special fiber, either $p$ does not divide $q$, in which case $D$ is in the branch locus, or else we obtain $p$ copies of $D$ branched along the points where $D$ meets the branch locus, and the multiplicity is reduced by $1/p$.
\end{remark}

\begin{example}(Example \ref{melodyexample}, Continued) \label{ex-semistable-2}
We must make two base changes of order 2. Starting with Figure \ref{fig-semistable-cd}(d) above, we see that $l_1$, $l_2$, and $C$ are in the branch locus. The curves $e_1'$ and $e_2'$ are replaced by the double cover of each of them branched at 2 points, which is again a rational curve. We continue to call these $e_1'$ and $e_2'$, and they each have multiplicity 2. Then, $e_1$ and $e_2$ are disjoint from the branch locus, so each one is replaced by two disjoint rational curves. The result is depicted in Figure \ref{fig-semistable-ef}(e). In the second base change of order 2, all components except $e_1'$ and $e_2'$ are in the branch locus. The curves $e_1'$ and $e_2'$ each meet the branch locus in 4 points, which, by the Riemann-Hurwitz theorem, means they will be replaced by genus 1 curves. The result is depicted in Figure \ref{fig-semistable-ef}(f).

\begin{figure}
\begin{center}
\begin{tabular}{cc}
 \includegraphics[height=1.6in]{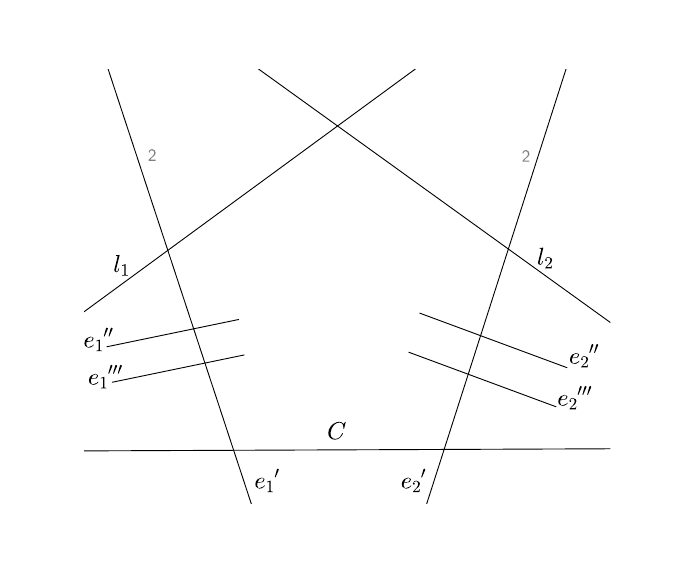} &   \includegraphics[height=1.6in]{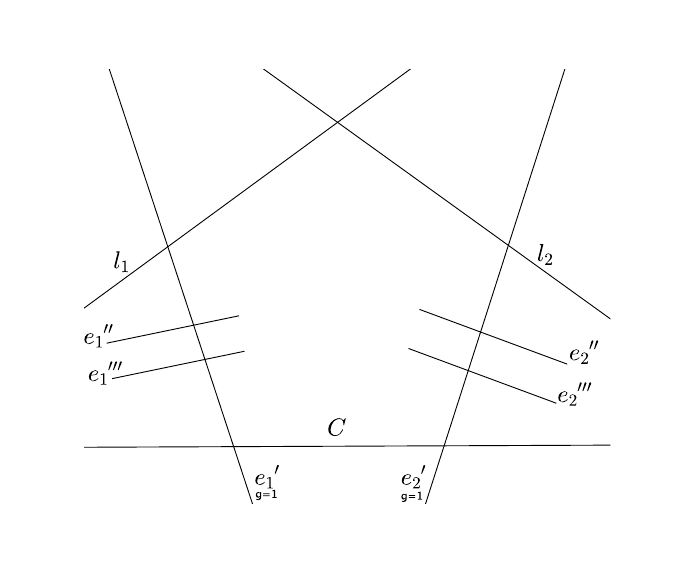} \\
 (e) & (f) \\ \\ 
\end{tabular}
\caption{Semistable reduction in Example \ref{ex-semistable-2}.}
\label{fig-semistable-ef}
\end{center}
\end{figure}
\end{example}
 
 The last step is to blow down all rational curves which meet the rest of the fiber exactly once, depicted in Figure \ref{melodyfinal}. This gives us a semistable model of $X$.

\begin{figure}
\begin{center}
 \includegraphics[height=1.7in]{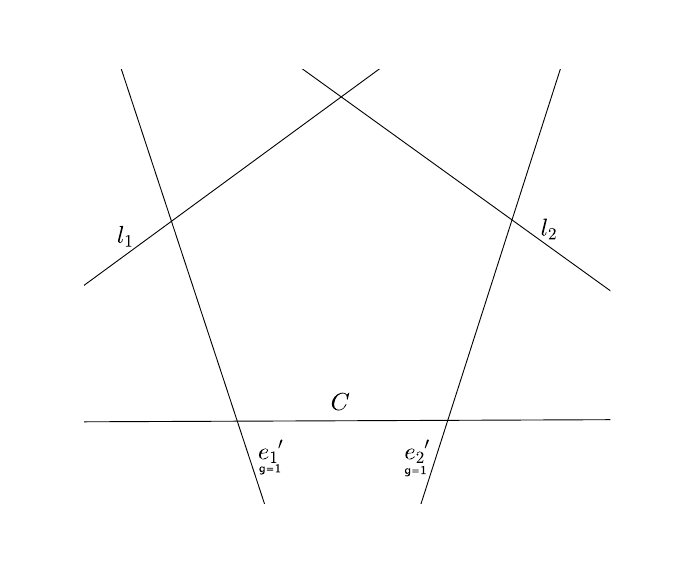}
 \caption{The special fiber of the semistable model obtained in Example \ref{melodyexample}}
 \label{melodyfinal}
 \end{center}
 \end{figure}

\subsection{Weighted Metric Graphs}
From a faithful tropicalization, the abstract tropicalization of a curve $X$ can be obtained simply by taking the minimal skeleton of $\Trop(X)$.
Given a semistable model $\mathcal{X}$ of $X$, this coincides with the \emph{dual graph} of $\mathcal{X}_k$ \cite{pbr}.
\begin{definition}
 Let $C_1, \ldots, C_n$ be the irreducible components of $\mathcal{X}_k$, the special fiber of a semistable model of $X$. 
The \emph{dual graph} of $\mathcal{X}_k$, $G$, is defined with vertices $v_i$ corresponding to the components $C_i$, where we set $w(v_i) = g(C_i)$. There is an edge $e_{ij}$ between $v_i$ and $v_j$ if  the corresponding components $C_i$ and $C_j$ intersect in a node $q$. Then, the completion of the local ring $\mathcal{O}_{\mathcal{X},q}$ is isomorphic to $R[[x,y]]/(xy-f)$, where $R$ is the valuation ring of $\Bbbk$, and $f \in m$, the maximal ideal of $R$. Then, we define $l(e_{ij}) = v(f)$.
\end{definition}

 \begin{example}(Example \ref{melodyexample}, Continued)
 \label{example7}
 Taking the dual graph of the semistable model found in the previous example, we obtain a cycle with two vertices of weight 1. By \cite[Theorem 5.24]{bpr}, the cycle that we observed in the embedded tropicalization actually gave a faithful tropicalization. This implies that the metric graph is as depicted in Figure \ref{melodyexfig}.
 \begin{figure}
 \begin{center}
 \includegraphics[height = 1in]{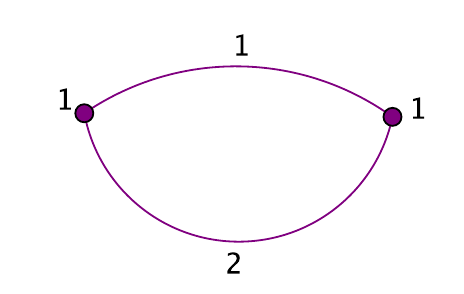}
 \caption{This is the metric graph from Example \ref{example7}. Since the vertices corresponding to $l_1$ and $l_2$ each have valence 2, we do not depict these in the model in this figure.}
 \label{melodyexfig}
 \end{center}
 \end{figure}
 \end{example}

\section{Period Matrices of Weighted Metric Graphs}
\label{periodmatrix}
Recall that a \tropcurve $\Gamma = (G,w,l)$ is a metric graph  with a model $G$, a function $w$ on $V(G)$ assigning nonnegative weights to the vertices, and a function $l$ on $E(G)$ assigning positive lengths to the edges. Given $\Gamma$, we describe a procedure to compute its period matrix, following \cite{mz, FMS, chan12}.

Fix an orientation of the edges of $G$. For any $e \in E(G)$, denote the source vertex by $s(e)$ and the target vertex by $t(e)$. Let $A$ be $\mathbb{R}$ or $\mathbb{Z}$. The spaces of 0-chains and 1-chains of $G$ with coefficients in $A$ are defined as
\begin{align}
C_0(G,A) = \left \{\sum_{v \in V(G)} a_v v\ |\ a_v \in A \right \}\,, \hspace{5mm}
C_1(G,A) = \left \{\sum_{e \in E(G)} a_e e\ |\ a_e \in A \right \}\,.
\end{align}

The module $C_1(G,\mathbb{R})$ is equipped with the inner product
\begin{align}
\label{innerproduct}
\left\langle \sum_{e\in E(G)} a_e e, \sum_{e\in E(G)} b_ee  \right\rangle = \sum_{e \in E(G)} a_eb_el(e)\,.
\end{align}

The \emph{boundary map} $\partial: C_1(G,A) \rightarrow C_0(G, A)$ acts linearly on 1-chains by mapping an edge $e$ to $t(e) - s(e)$. The kernel of this map is the first homology group $H_1(G,A)$ of $G$, whose rank is 
\begin{equation}
g(G)= |E(G)| -|V(G)|+1.
\end{equation}
 Let $|w| = \sum_{v\in V(G)} w(v)$, and let $g$ be the genus of $\Gamma$, defined as $g(G)+|w|$. Consider the positive semidefinite form $Q_\Gamma$ on $H_1(G,\mathbb{R}) \oplus \mathbb{R}^{|w|}$, which vanishes on the second summand $\mathbb{R}^{|w|}$ and is defined on $H_1(G,\mathbb{R})$ by
\begin{align}
Q_\Gamma \left(\sum_{e\in E(G)} \alpha_e e\right) = \sum_{e\in E(G)} \alpha_e^2 l(e)\,.
\end{align}

\begin{definition}
Let $\omega_1,\ldots, \omega_{g(G)}$ be a basis of $H_1(G,\mathbb{Z})$. Then, we obtain an identification of the lattice $H_1(G,\mathbb{R})\oplus \mathbb{R}^{|w|}$ with $\mathbb{R}^g$. Hence, we may express $Q_\Gamma$ as a positive semidefinite $g\times g$ matrix, called the \emph{period matrix} of $\Gamma$. Choosing a different basis gives another matrix related by an action of $GL_g(\mathbb{Z})$.
\end{definition}

To find the period matrix, first fix an arbitrary orientation of the edges of $G$. We then pick a spanning tree $T$ of $G$. Label the edges such that $e_1,\ldots, e_{g(G)}$ are not in $T$, and $e_{g(G)+1},\ldots, e_m$ are in $T$, where $m=|E(G)|$. Then $T\cup\{e_i\}$, for $1\leq i\leq g(G)$, contains a unique cycle $\omega_i$ of $G$. The cycles $\omega_1,\ldots,\omega_{g(G)}$ form a cycle basis of $G$.

We traverse each cycle $\omega_i$ according to the direction specified by $e_i$. We compute a row vector $b_i$ of length $m$, representing the direction of edges of $G$ in this traversal. For each edge $e_j$ in $E(G)$, let the $j$-th entry of $b_i$ be $1$ if $e_j$ is in the correct orientation in the cycle, $-1$ if it is in the wrong orientation, and 0 if it is not in the cycle. Let $B$ be the $g(G)\times m$ matrix whose $i$-th row is $b_i$. The matrix $B$ has an interpretation in matroid theory as a totally unimodular matrix representing the cographic matroid of $G$ \cite{oxley}.

Suppose that all vertices have weight zero, such that $g(G) = g$. Let $D$ be the $m\times m$ diagonal matrix with entries $l(e_1),\ldots,l(e_m)$. Then the period matrix is given by $Q_\Gamma = B D B^T$. If we label the columns of $B$ by $v_1,\ldots,v_m$, the period matrix equals 
\begin{align} \label{eqn-period-matrix}
Q_\Gamma = l(e_1) v_1v_1^T + \cdots + l(e_m)v_mv_m^T\,.
\end{align}

Thus the cone of all matrices that are period matrices of $G$, allowing the edge lengths to vary, is the rational open polyhedral cone
\begin{align}\label{cone}
\sigma_{G} = \mathbb{R}_{>0}\langle v_1v_1^T,\ldots, v_mv_m^T\rangle\,.
\end{align}

If $\Gamma$ has vertices of nonzero weight, the period matrix is given by the construction above with $g-g(G)$ additional rows and columns with zero entries.

\begin{example}
\label{completegraph}
Consider the complete graph on 4 vertices in Figure \ref{completefig}. 

\begin{figure}[h]
\begin{center}
\begin{tikzpicture}
\begin{scope}[thick, decoration={markings, mark=at position 0.5 with {\arrow{latex}}}]
\draw[postaction={decorate}] (2,2.5) --(0,0);
\draw[postaction={decorate}] (0,0) -- (4,0);
\draw[postaction={decorate}] (4,0) -- (2,2.5);
\draw[postaction={decorate}] (0,0) -- (2,0.9);
\draw[postaction={decorate}] (2,0.9) -- (2,2.5);
\draw[postaction={decorate}] (4,0) -- (2,0.9);
\filldraw (0,0) circle (1.5pt) (2,2.5) circle (1pt) (2,0.9) circle (1pt) (4,0) circle (1pt);
\draw (1,1.6) node{2};
\draw (0.7,1.3) node[red]{$e_1$};
\draw (2.3,1.5) node{13};
\draw (1.7, 1.5) node[red]{$e_2$};
\draw (3,1.6) node{3};
\draw (3.3, 1.3) node[red]{$e_6$};
\draw (1.2,0.3) node{7};
\draw (1.3,0.8) node[red]{$e_3$};
\draw (2.8,0.3) node{11};
\draw (2.9,0.7) node[red]{$e_4$};
\draw (2,-0.2) node{5};
\draw (2, 0.2) node[red]{$e_5$};
\end{scope}
\end{tikzpicture}
\caption{The metric graph and edge orientation used in Example \ref{completegraph}.}
\label{completefig}
\end{center}
\end{figure}
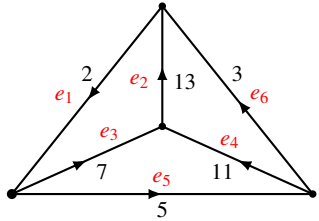

We indicate in the figure an arbitrary choice of the edge orientations, and we choose the spanning tree consisting of the edges $T = \{e_2,e_3,e_4\}$. This corresponds to the cycle basis
$\omega_1 = e_1 + e_3 + e_2 $, 
$\omega_2 = -e_3 + e_5 + e_4$, and 
$\omega_3 = -e_2 - e_4 + e_6$.
Next, we compute the matrix $B$ as
\begin{align}
B = 
\begin{pmatrix}
1 & 1 & 0 & 1 & 0 & 0\\
0 & -1 & 1 & 0 & 1 & 0\\
-1 & 0 & -1 & 0 & 0 & 1
\end{pmatrix}\,.
\end{align}
Let $D$ be the $6\times 6$ diagonal matrix with entries $13,7,11,2,5,3$ along the diagonal. The period matrix is then
\begin{align} \label{ex-period-mat}
Q_\Gamma = B D B^T =
\begin{pmatrix} 22 & -7 & -13 \\ -7 & 23 & -11 \\ -13 & -11 & 27\end{pmatrix}\,.
\end{align}

\end{example}

In the next section, we will use period matrices to define and study tropical Jacobians of curves as principally polarized tropical abelian varieties.

\section{Tropical Jacobians}
\label{jacobian}

Let $\tilde{S}^g_{\geq 0}$ be the set of $g\times g$ symmetric positive semidefinite matrices with rational nullspace, meaning that their kernels have bases defined over $\mathbb{Q}$. The group $GL_g(\mathbb{Z})$ acts on $\tilde{S}^g_{\geq 0}$ by $Q\cdot X = X^T Q X$ for all $X \in GL_g(\mathbb{Z})$, $Q\in \tilde{S}^g_{\geq 0}$.

We define a \emph{tropical torus} of dimension $g$ as a quotient $X=\mathbb{R}^g/ \Lambda$, where $\Lambda$ is a lattice of rank $g$ in $\mathbb{R}^g$. A \emph{polarization} on $X$ is given by a quadratic form $Q$ on $\mathbb{R}^g$. Following \cite{FMS, chan12}, we call the pair $(\mathbb{R}^g/\Lambda, Q)$ a \emph{principally polarized tropical abelian variety (pptav)}, when $Q \in \tilde{S}^g_{\geq 0}$.

Two pptavs are isomorphic if there is some $X\in GL_g(\mathbb{R})$ that maps one lattice to the other, and acts on one quadratic form to give the other. We can choose a representative of each isomorphism class in the form $(\mathbb{R}^g/\mathbb{Z}^g, Q)$, where $Q$ is an element of the quotient of $\tilde{S}^g_{\geq 0}$ by the action of $GL_g(\mathbb{Z})$. The points of this space are in bijection with the points of the \emph{moduli space of principally polarized tropical abelian varieties}, which we denote by $A_g^{\tr}$. We will describe the structure of $A_g^{\tr}$ in Section \ref{section-schottky}.

\begin{definition}
The \emph{tropical Jacobian} of a curve is $(\mathbb{R}^g/\mathbb{Z}^g, Q)$, where $Q$ is the period matrix of the curve. 
\end{definition}
The tropical Jacobian is our primary example of a principally polarized tropical abelian variety. The period matrix induces a Delaunay subdivision of $\mathbb{R}^g$, which has an associated Voronoi decomposition giving the \emph{tropical theta divisior} of the tropical Jacobian. We describe this in more detail below.

Given $Q \in \tilde{S}^g_{\geq 0}$, consider the map
\begin{equation}
l _Q: \ZZ ^g \longrightarrow \ZZ ^g \times \RR ,\ \ \  x \mapsto (x, x^T Q x)\,. 
\end{equation}
Take the convex hull of image of $l_Q$ in $\RR ^{g}\times \RR \cong \RR^{g+1}$. By projecting down the lower faces via the morphism $\RR^{g+1}\longrightarrow \RR ^g$ which forgets the last coordinate, we obtain a periodic dicing of the lattice $\ZZ ^g\subset \RR ^g$, called the \emph{Delaunay subdivision} $\Del(Q)$ of $Q$. This operation corresponds, naively speaking, to looking at the polyhedron from below and recording on the lattice only the faces that we see. This is an infinite and periodic analogue of taking the regular subdivision of a polytope induced by weights on the vertices.

\begin{remark} There is in the literature much discordance about the
  spelling of the name Delaunay. The source of such discordance lies
  in the fact that Boris Nicolaevich Delaunay was a Russian
  mathematician, hence the correct transliteration of his last name
  from Cyrillic is much debated (since he himself used the two
  versions ``Delaunay" and ``Delone" when he signed his papers). One
  must notice, however, that the name itself is of French origin:
  indeed, Boris Delaunay got his last name from the French Army
  officer De Launay, who was captured in Russia during Napoleon's
  invasion of 1812 and, after marrying a Russian noblewoman, settled
  down in Russia (see \cite{Roz}). Hence we decided to use the French
  transliteration, it being closer to the original name.
\end{remark}

\begin{example} 
\label{ex9}
Consider the matrix $\begin{bmatrix}1&0\\0&0\end{bmatrix}$. The function $l_Q: \ZZ ^2 \rightarrow \ZZ^2 \times \RR$ is given by $(x,y) \mapsto (x,y,x^2)$. If we now take the convex hull of the points in the image of $\l_Q$, we obtain the following picture on the left, together with the Delaunay subdivision on the right.

\begin{figure}
\begin{center}
\includegraphics[height = 1.7in]{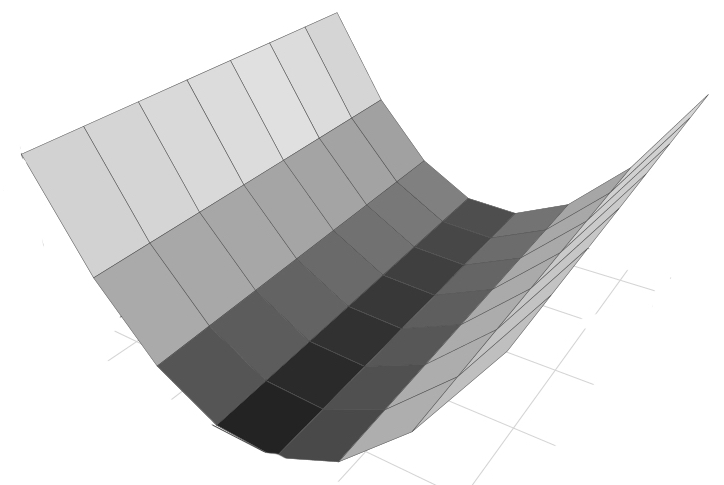}
\includegraphics[height = 1.8 in]{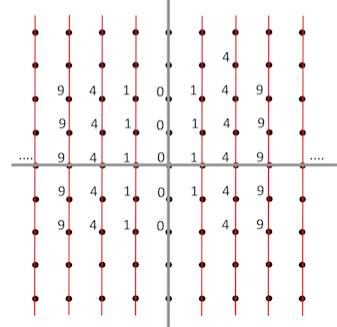}
\caption{The weight function induced by the quadratic form in Example \ref{ex9} on the left, and the corresponding Delaunay subdivision on the right.}
\end{center}
\end{figure}
\end{example}

Given a Delaunay decomposition, one can consider the dual decomposition, called the {\it Voronoi decomposition}. This is illustrated in Figure \ref{fig-voronoi-r2} for $g=2$.
\begin{figure}
\begin{center}
\includegraphics[height = 1.50 in]{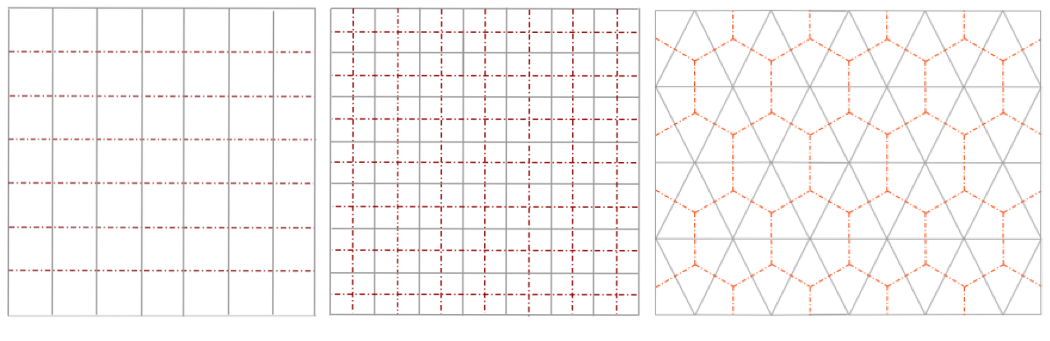}
\caption{Delaunay decompositions of $\mathbb{R}^2$ (solid lines) and their associated Voronoi decompositions (dotted lines).}
\label{fig-voronoi-r2}
\end{center}
\end{figure}
The Voronoi decomposition corresponds to the \emph{tropical theta divisor} associated to a pptav $(\mathbb{R}^g/\Lambda, Q)$, which is the tropical hypersurface in $\mathbb{R}^g$ defined by the \emph{theta function} 
\begin{align}
\Theta(x) = \max_{\lambda \in \Lambda} \left\{ \lambda^T Q x - \frac{1}{2}\lambda^T Q \lambda\right\}\,,\hspace{5mm} x\in\mathbb{R}^g\,.
\end{align}

It is possible to give a more manageable description of the tropical theta divisor in virtue of Theorem \ref{MZ}, as we are about to explain. 
Let~$\Gamma=(G,w,l)$ be a \tropcurve, and let $p_0\in \Gamma$ be a fixed basepoint. Let $\omega_1, \ldots, \omega_g$ be a basis of $H_1(G,\mathbb{Z})$. For any point $p$ in $\Gamma$, let $c(p)= \sum_i a_i e_i$ describe any path from $p_0$ to $p$. Then, take the inner product (defined in Equation \ref{innerproduct}) of $c(p)$ with each element of the cycle basis to obtain a point of $\mathbb{R}^g / \Lambda$, given by
\begin{equation}
(\langle c(p), \omega_1 \rangle, \ldots, \langle c(p), \omega_g \rangle).
\end{equation}
This does not depend on the choice of path from $p_0$ to $p$.
By the identification with $\mathbb{R}^g / \mathbb{Z}^g$ induced by the choice of cycle basis, this defines a point $\mu(p)$ of the tropical Jacobian. We may extend this map linearly so that it is defined on all divisors on $\Gamma$. By a \emph{divisor} on $\Gamma$, we mean a finite formal integer linear combination of points in $\Gamma$. Then the map $\mu$ is called the \emph{tropical Abel-Jacobi map} \cite{mz}. 

Given a divisor $D = \sum_ia_ip_i$, where $a_i \in \mathbb{Z}$ and $p_i \in \Gamma$, define the \emph{degree} of $D$ as $\sum_ia_i$. We say that $D$ is \emph{effective} if $a_i\geq 0$ for all $i$. Let $W_{g-1}$ be the image of degree $g-1$ effective divisors under the tropical Abel-Jacobi map.

\begin{theorem}[Corollary 8.6, \cite{mz}] \label{MZ}
The set $W_{g-1} $ is the tropical theta divisor up to translation.
\label{mzthm}
\end{theorem}

\begin{example}[Example \ref{completegraph}, Continued]
Delaunay subdivisions also arise in many other branches of mathematics, for example in lattice packing or covering problems. In this context, Sikiri\'c wrote a \emph{GAP} \cite{GAP4} software package \emph{polyhedral} \cite{gap}. Using this package, we compute that the Delaunay subdivision of the quadratic form in Equation \ref{ex-period-mat} is given by six tetrahedra in the unit cube, all of which share the great diagonal as an edge. We also compute using \emph{polyhedral} the Voronoi decomposition dual to this Delaunay subdivision, which gives a tiling of $\mathbb{R}^3$ by permutohedra as illustrated in Figure \ref{figure-tiling-permutohedra}. This is the tropical theta divisor, with $f$-vector $(6,12,7)$. In Figure \ref{permutohedron}, we illustrate the correspondence described by Theorem \ref{mzthm} between $W_2$ and the tropical theta divisor.

\begin{figure}
\begin{center}
\includegraphics[height = 3.5cm]{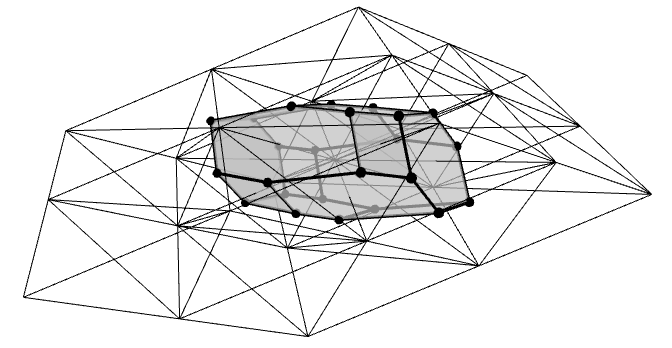}
\includegraphics[height = 3cm]{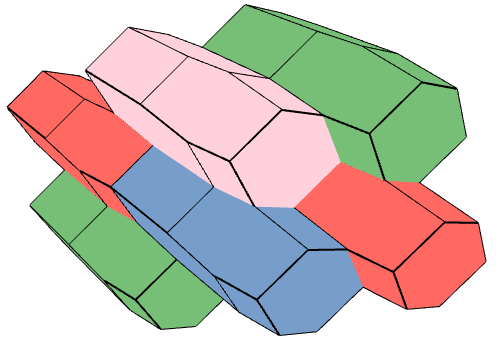}
\caption{The left figure shows the Delaunay subdivision by tetrahedra and a dual permutohedron in grey. The right figure illustrates a tiling of $\mathbb{R}^3$ by permutohedra. The polytopes were computed using the \emph{polyhedral} package of \emph{GAP} \cite{gap} and the figures were created using \emph{polymake} \cite{polymake}.}
\label{figure-tiling-permutohedra}
\end{center}
\end{figure}

\begin{figure}
\begin{center}
\includegraphics[height =4 in]{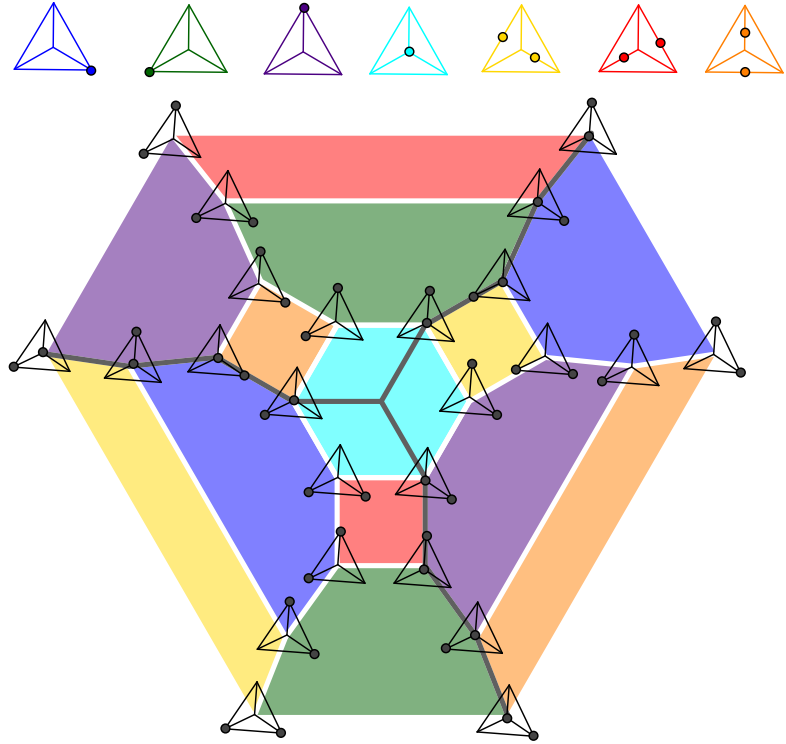}
\caption{Each vertex of the permutohedron corresponds to a divisor supported on the vertices of $\Gamma$. The square faces correspond to divisors supported on the interiors of edges of $\Gamma$ which do not meet in a vertex. Each hexagonal face corresponds to divisors which are supported on edges of $\Gamma$ which are adjacent to a fixed vertex. Then, the edges correspond to keeping one point of the divisor fixed, and moving the other point along an edge of $\Gamma$. The grey curve depicted above represents the embedding of $\Gamma$ into its Jacobian under the Abel-Jacobi map, which, under the identifications, is again $K_4$.}
\label{permutohedron}
\end{center}
\end{figure}
\end{example}

\section{Tropical Schottky Problem}
\label{section-schottky}
We describe the structure of the moduli space $A_g^{\tr}$ in this section, using Voronoi reduction theory. Given a Delaunay subdivision $D$, define the set of matrices that have $D$ as their Delaunay subdivision to be
\begin{align} \label{def-sec-cone}
\sigma_D = \{Q\in \tilde{S}^g_{\geq 0}\,|\, \Del(Q) = D\}\,.
\end{align}

The \emph{secondary cone} of $D$ is the Euclidean closure $\overline{\sigma_D}$ of $\sigma_D$ in $\mathbb{R}^{\binom{g+1}{2}}$, and is a closed rational polyhedral cone. There is an action of $GL_g(\mathbb{Z})$  on the set of secondary cones, induced by its action on $\tilde{S}^g_{\geq 0}$.

\begin{theorem}[\cite{Voronoi1908}]
The set of secondary cones forms an infinite polyhedral fan whose support is $\tilde{S}^g_{\geq 0}$, known as the \emph{second Voronoi decomposition}. There are only finitely many $GL_g(\mathbb{Z})$-orbits of this set of secondary cones.
\end{theorem}

By this theorem, we can choose Delaunay subdivisions $D_1,\ldots, D_k$ of $\mathbb{R}^g$, such that the corresponding secondary cones are representatives for $GL_g(\mathbb{Z})$-equivalence classes of secondary cones. The moduli space $A_g^{\tr}$ is a \emph{stacky fan} whose cells correspond to these classes \cite{chan12, FMS}. More precisely, for each Delaunay subdivision $D$, consider the stabilizer
\begin{align}
\Stab(\sigma_D) = \{ X \in GL_g(\mathbb{Z})\,|\, \sigma_D \cdot X = \sigma_D\}\,.
\end{align}
Define the cell
\begin{align}
C(D) = \overline{\sigma_D} / \Stab(\sigma_D)
\end{align}
as the quotient of the secondary cone by the stabilizer. Then we have
\begin{align}
A_g^{\tr} = \bigsqcup_{i=1}^k C(D_i) / \sim,
\end{align}
where we take the disjoint union of the cells $C(D_1),\ldots, C(D_k)$ and quotient by the equivalence relation $\sim$ induced by $GL_g(\mathbb{Z})$-equivalence of matrices in $\tilde{S}^g_{\geq 0}$, which corresponds to gluing the cones.

\begin{example}
In genus two, we can choose the Delaunay subdivisions $D_1,\ldots, D_4$ as shown in Figure \ref{Delaunay2}. These have the property that their secondary cones give representatives for $GL_g(\mathbb{Z})$-equivalence classes of secondary cones.

\begin{figure}[h]
\begin{center}
\includegraphics[height = 0.9 in]{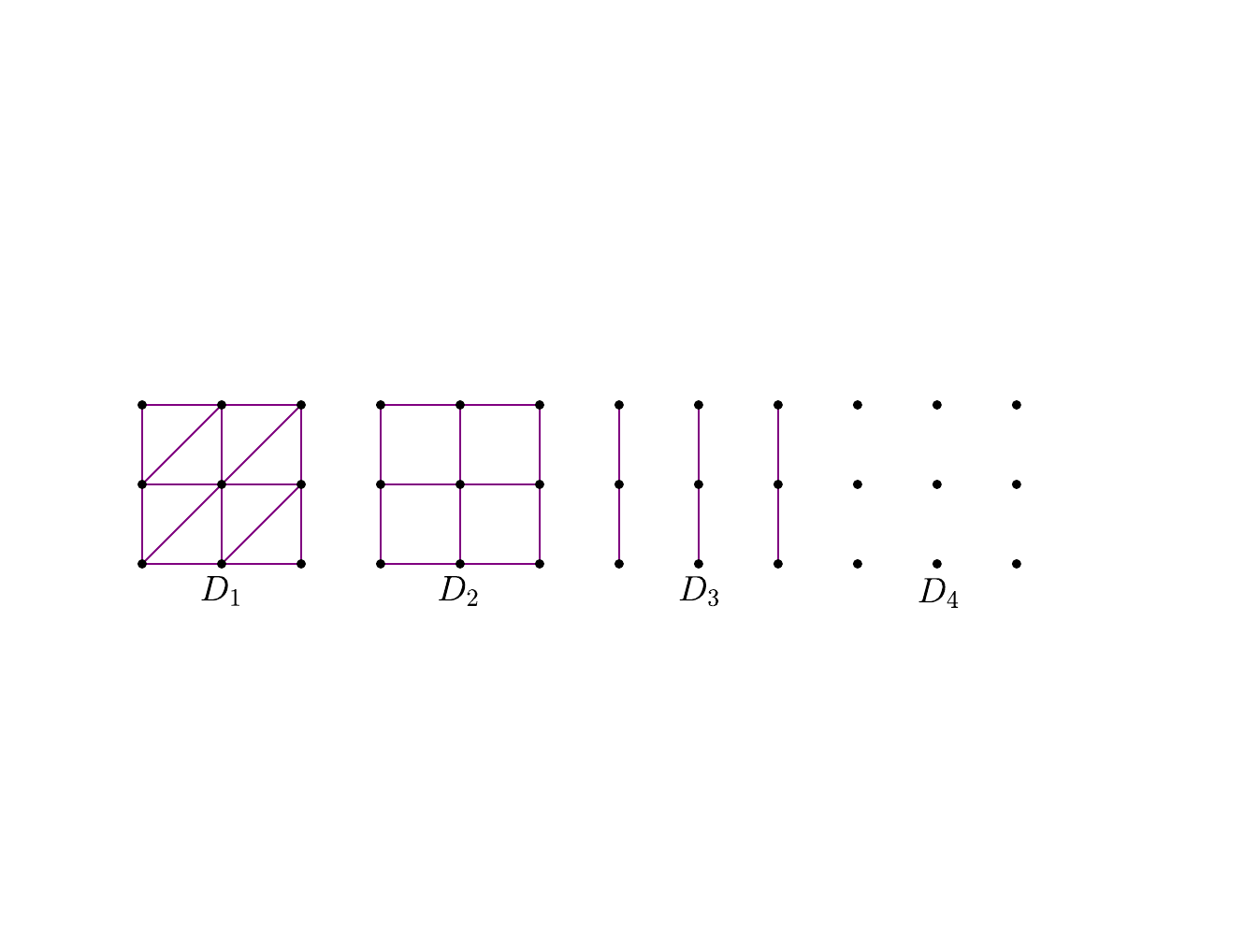}
\caption{Delaunay subdivisions for $g = 2$.}
\label{Delaunay2}
\end{center}
\end{figure}

The corresponding secondary cones are as follows.

\begin{align}
\overline{\sigma_{D_1}} &= \left\{\begin{bmatrix}a+c&-c\\-c&b+c\end{bmatrix}\,:\, a, b, c \in\mathbb{R} \right\}\,, \\
\overline{\sigma_{D_2}} &= \left\{\begin{bmatrix}a&0\\0&b\end{bmatrix}\,:\, a, b \in\mathbb{R} \right\}\,, \\
\overline{\sigma_{D_3}} &= \left\{\begin{bmatrix}a&0\\0&0\end{bmatrix}\,:\, a \in\mathbb{R} \right\}\,, \\
\overline{\sigma_{D_4}} &= \left\{\begin{bmatrix}0&0\\0&0\end{bmatrix}\right\}\,.
\end{align}
\end{example}

The \emph{tropical Torelli map} $t_g^{\tr} : M_g^{\tr} \rightarrow A_g^{\tr}$ sends a \tropcurve of genus $g$ to its tropical Jacobian, which is the element of $A_g^{\tr}$ corresponding to its period matrix. The image of this map is called the \emph{tropical Schottky locus}, which has a characterization using matroid theory. 

Given a graph $G$, we can define a \emph{cographic matroid} $M^*(G)$ (see \cite{oxley} for an introduction to matroid theory). $M^*(G)$ is representable by a totally unimodular matrix, constructed as the matrix $B$ in Section \ref{periodmatrix}. The cone $\sigma_G$ defined in Equation \ref{cone} is a secondary cone in $\tilde{S}^g_{\geq 0}$. The $GL_g(\mathbb{Z})$-equivalence class of $\sigma_{G}$ is independent of the choice of totally unimodular matrix representing $M^*(G)$. Hence we can associate to $M^*(G)$ a unique cell $C(M^*(G))$ of $A_g^{\tr}$, corresponding to this equivalence class of secondary cones. 

A matroid is \emph{simple} if it has no loops and no parallel elements. We define the following stacky subfan of $A_g^{\tr}$ corresponding to simple cographic matroids,
\begin{align}
  A_g^{\cogr} = \left\{ C(M)\, :\, M \mbox{ simple cographic matroid of rank } \leq g\right\}\,.
\end{align}
The image of the tropical Torelli map $t_g^{\tr}$ is $A_g^{\cogr}$ \cite{FMS, chan12}, and we call $A_g^{\cogr}$ the \emph{tropical Schottky locus}. When $g\leq 3$, $A_g^{\cogr} = A_g^{\tr}$, hence every element of $\tilde{S}^g_{\geq 0}$ is a period matrix of a \tropcurve. But when $g\geq 4$, this inclusion is proper. For example, $A_4^{\cogr}$ has 25 cells while $A_4^{\tr}$ has 61 cells, and $A_5^{\cogr}$ has 92 cells while $A_5^{\tr}$ has 179433 cells, according to the computations in \cite{chan12}.

From the classical perspective, the association $C\mapsto (J(C), \Theta)$, where $J(C)$ is the Jacobian of $C$ defined in the introduction, gives the \emph{Torelli map}

$$
t_g : \mathcal{M}_g \to \mathcal{A}_g\,.
$$

Here $ \mathcal{M}_g$ denotes the moduli space of smooth genus $g$ curves, while $\mathcal{A}_g$ denotes the moduli space of $g$-dimensional abelian varieties with a principal polarization. The content of Torelli's theorem is precisely the injectivity of the Torelli map, which in fact can be proved to be dominant for $g=2,3$. Its image, i.e. the locus inside the moduli space of principally polarized abelian varieties, is called the {\it Schottky locus} and its complete description required several decades of work by many. The injectivity of the Torelli morphism implies, in particular, that we can always reconstruct an algebraic curve from its principally polarized Jacobian.
The characterization of the Schottky locus was first worked out in genus 4 by Schottky himself and Jung (cf. \cite{SJ}) via theta characteristics. Many other approaches followed in history for higher genus: the work by Andreotti and Mayer (\cite{andreotti}), Matsusaka and Ran {\cite{matsusaka, ran}) and Shiota (\cite{shiota}) are worth mentioning. For extensive surveys see e.g. \cite{arbarelloschottky,gru}.

\

\section{...and Back}
\label{andback}

The process we have described so far produces the tropical Jacobian of a curve given its defining equations. Now, we discuss whether it is possible to take a principally polarized tropical abelian variety X in the tropical Schottky locus, and produce a curve whose tropical Jacobian is precisely X. We remark that several of the steps described in the previous sections are far from being one-to-one.
Indeed, many algebraic curves have the same abstract tropicalization; for example all curves with a smooth stable model tropicalize to a single weighted vertex.
 In the same fashion, the non-injectivity of the tropical Torelli map (see e.g. \cite{ FMS, chan12}) implies that the same positive semidefinite matrix can be associated to more than one weighted metric graph. The purpose of this section is therefore to construct an arbitrary curve with a given tropical Jacobian. 

\subsection{From Tropical Jacobians to Positive Semidefinite Matrices}

Let $(\RR ^g /\Lambda, \Theta)$ be a tropical Jacobian, and fix an isomorphism $\Lambda \cong \ZZ ^g$. The tropical theta divisor $\Theta$, as we remarked in Section \ref{jacobian}, is a Voronoi decomposition dual to a Delaunay subdivision $D$. 

We can describe $D$ by a collection of hyperplanes $\{H_1,\ldots, H_k\}$, such that the lattice translates by $\mathbb{Z}^g$ of these hyperplanes cut out the polytopes in $D$. Following \cite[Fact 4.1.4]{meloviviani} and \cite{ER}, we can choose these hyperplanes with normal vectors $v_1,\ldots,v_k \in\mathbb{R}^g$ such that the matrix with $v_1,\ldots,v_k$ as its columns is simple unimodular. The secondary cone of $D$ is then 
\begin{align}
\sigma_D = \RR _{>0} \langle v_1v_1^T,\ldots,v_kv_k^T\rangle\,.
\end{align}
Thus any quadratic form lying in the positive span of the rank one forms $v_i v_i^T$, for $i=1,...,k$, will have Delaunay subdivision $D$. In particular, we can take
\begin{align}
Q = v_1v_1^T +\cdots+v_k v_k^T\,.
\end{align}

\subsection{From Positive Semidefinite Matrices to Weighted Metric Graphs}

Fix $g\geq 0$, and let $Q$ be a $g\times g$ matrix in $\tilde{S}^g_{\geq 0}$. If $Q$ is not positive definite, we can do a change of basis such that $Q$ has a $g'\times g'$ postive definite submatrix and remaining entries zero. This corresponds to adding $|w| = g-g'$ weights on the vertices of the graph, which can be done arbitrarily as long as every weight zero vertex has degree at least 3. Hence without loss of generality we assume that $Q$ is positive definite. Our goal is to first determine if $Q$ corresponds to an element of the tropical Schottky locus, and if so, to find a \tropcurve that has $Q$ as its period matrix. 

First consider all combinatorial types of simple graphs with genus
less than or equal to $g$. We compute their corresponding secondary
cones, as defined in Equation \ref{cone}. Let $S$ be the set of these
secondary cones. We compute the secondary cone $\sigma_Q :=
\sigma_{\Del(Q)}$ of $Q$, as defined in Equation \ref{def-sec-cone},
which can be done using \emph{polyhedral} \cite{gap}. The underlying
theory is described in \cite{schuermann}. We then check if
$\sigma_Q$ is $GL_g(\mathbb{Z})$-equivalent to any cone in $S$. This can be done using \emph{polyhedral} \cite{gap} with external calls to the program \emph{ISOM} by Plesken and Souvignier \cite{ps95, ps97}, see \cite[Section 4]{sgsw} for the implementation details. 

If $\sigma_Q$ is not equivalent to any cone in $S$, then $Q$ is not the period matrix of a \tropcurve. Otherwise, let $\sigma$ be a cone in $S$ that is in the same $GL_g(\mathbb{Z})$-equivalence class as $\sigma_Q$, and let $G$ be the graph from which we computed $\sigma$. Let $X\in GL_g(\mathbb{Z})$ such that $X$ maps $\sigma_Q$ to $\sigma$, then $X$ maps $Q$ to a matrix $Q'$ in $\sigma$.

From Equation \ref{cone}, we can write $\sigma_G =
\mathbb{R}_{>0}\langle v_1v_1^T,\ldots, v_m v_m^T\rangle$, and so $Q'
= \alpha_1 v_1v_1^T + \cdots + \alpha_m v_mv_m^T$ for some
$\alpha_1,\ldots,\alpha_m\in\mathbb{R}_{>0}$. Then $Q'$ is the period
matrix of $G$ with edge lengths $\alpha_1,\ldots,\alpha_m$, by
Equation \ref{eqn-period-matrix}. Since the transformation $X\in
GL_g(\mathbb{Z})$ corresponds to a different choice of cycle basis, we
have constructed a metric graph with $Q$ as its period matrix.

\begin{example}
Consider the positive definite matrix
\begin{align}
Q =\begin{pmatrix}
17 & 5 & 3 & 5\\
5 & 19 & 7 & 11\\
3 & 7 & 23 & 16 \\
5 & 11 & 16 & 29
\end{pmatrix}\,. \label{eqn-ex-trop-schottky}
\end{align}
Using \emph{polyhedral}, we
compute that $\sigma_Q$ is $GL_4(\mathbb{Z})$-equivalent to the cone
$\sigma\in S$ corresponding to the weighted metric graph in Figure
\ref{fig-recover-trop-schottky}, via the transformation
\begin{align}
X = \begin{pmatrix}
  0 & 0 & 0 & 1 \\
  1 & 0 & 0 & 0 \\
  0 & 1 & 1 & 0 \\
  -1 & -1 & 0 & 0
\end{pmatrix}\,,
\hspace{1.5cm}
Q' = X^TQX = \begin{pmatrix}
  26 & 9 & -9 & 0 \\
  9 & 20 & 7 & -2 \\
  -9 & 7 & 23 & 3 \\
  0 & -2 & 3 & 17
\end{pmatrix}\,. \label{eqn-transform-trop-schottky}
\end{align}

Hence $Q$ is in the tropical Schottky locus, and $Q'$ is the period
matrix of the metric graph in Figure \ref{fig-recover-trop-schottky},
with the cycle basis consisting of the cycles $e_2+e_6-e_3$,
$-e_1+e_2+e_7-e_4$, $-e_1+e_3+e_8-e_5$, and $e_4+e_9-e_5$. We compute
the edge lengths by expressing $Q'$ as a linear combination of the
extreme rays of $\sigma$.

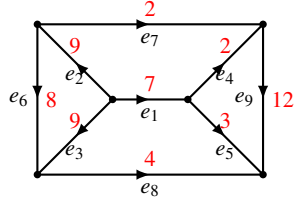
\begin{figure}[h]
\begin{center}
\begin{tikzpicture}
\begin{scope}[thick, decoration={markings, mark=at position 0.5 with {\arrow{latex}}}]
\filldraw (0,0) circle (1pt) (0,2) circle (1pt) (1,1) circle (1pt)
(2,1) circle (1pt) (3,0) circle (1pt) (3,2) circle (1pt);
\draw[postaction={decorate}] (1,1)--node[above,red]{7}node[below]{$e_1$}(2,1);
\draw[postaction={decorate}] (1,1)--node[above,red]{9}node[below]{$e_2$}(0,2);
\draw[postaction={decorate}] (1,1)--node[above,red]{9}node[below]{$e_3$}(0,0);
\draw[postaction={decorate}] (2,1)--node[above,red]{2}node[below]{ $e_4$}(3,2);
\draw[postaction={decorate}] (2,1)--node[above,red]{3}node[below]{$e_5$}(3,0);
\draw[postaction={decorate}] (0,2)--node[right,red]{8}node[left]{$e_6$}(0,0);
\draw[postaction={decorate}] (0,2)--node[above,red]{2}node[below]{$e_7$}(3,2);
\draw[postaction={decorate}] (0,0)--node[above,red]{4}node[below]{$e_8$}(3,0);
\draw[postaction={decorate}] (3,2)--node[right,red]{12}node[left]{$e_9$}(3,0);
\end{scope}
\end{tikzpicture}
\caption{Weighted metric graph with Equation \ref{eqn-ex-trop-schottky} as a
  period matrix, with edge lengths in red.}
\label{fig-recover-trop-schottky}
\end{center}
\end{figure}
\end{example}

As of the time of writing, this algorithm is impractical for genus
greater than 5, as the classification of Delaunay subdivisions is
known only up to dimension 5 \cite{sgsw}. For more details about the
genus 4 case, refer to future work with M. Kummer and B. Sturmfels.

\subsection{From Weighted Metric Graphs to Algebraic Curves}

 Lastly, we wish to take a weighted metric graph $\Gamma$ and produce
 equations defining a curve which tropicalizes to $\Gamma$. Any
 weighted metric graph $\Gamma$ arises through tropicalization
 \cite[Theorem 1.2.1]{ACP}. Given a smooth curve $X$ with $\Gamma$ as
 its tropicalization, there exists a rational map $f : X \rightarrow
 \mathbb{P}^3$ such that the restriction of $\trop(f)$ to the skeleton
 $\Gamma$ is an isometry onto its image \cite[Theorem
   8.2]{jacobians}. Since $f$ is not required to be a closed
 immersion, however, this does not necessarily give a faithful
 tropicalization, see \cite[Remark 8.5]{jacobians}.

 The paper \cite{CFPU} also studies this question for a specific class
 of metric graphs. They give a method for producing curves over $\mathbb{C}((t^{1/l}))$, embedded in a toric scheme and with a faithful tropicalization to the input metric graph $\Gamma$. They start by defining a suitable nodal curve whose dual graph is a model for $\Gamma$, and use deformation theory to show that the nodal curve can be lifted to a proper, flat, semistable curve over $R$ with the nodal curve as its special fiber, which tropicalizes to $\Gamma$.
 
 We now describe a procedure for finding a nodal curve over $\mathbb{C}$ whose dual graph is a model for $\Gamma$. 
Let $G$ be a weighted stable graph of genus $g$ with $n$ infinite edges. Recall that a stable graph $G$ is a connected graph such that each vertex of weight zero has valence at least three. The dual graph of a stable curve is always a stable graph.

 The original idea for this procedure is due to Koll\'ar, cf. \cite{kol}, and works in a much more general setup. Suppose that the stable graph $\Gamma = (G , w, l)$ is such that:
\begin{enumerate}
\item For each vertex $v\in V(G)$, the weight $w(v)$ is of the form
\begin{equation}
w(v) = {d(v) - 1 \choose 2} 
\end{equation}
for some integer $d(v)$.
\item For each two vertices $v,w \in V(G)$, one has
\begin{equation}
| E(v,w) | \leq d(v)d(w) ,
\end{equation}

where $|E(v,w)|$ denotes the number of edges between $v$ and $w$. 
\end{enumerate}

Then every component of $G$ is realizable by a curve in $\PP^2$, and it is possible to achieve the right number of intersection points between every two components. More precisely, one can proceed as follows:
\begin{enumerate}
\item Label the vertices as $\{v_1,...,v_n\}=V(G)$. For each $i=1,...,n$, take a general smooth plane curve $C_i$ of degree $d_i = d(v_i)$.
\item We have now a reducible plane curve $C$, whose irreducible components are the curves $C_i$ of degree $d_i$ (and hence, by the genus degree formula, of genus $w(v_i)$). Any two components $C_i$ and $C_j$ will intersect in $d_id_j$ points, by B\'ezout's formula. We choose any $k_{ij} = d_id_j - |E(v_i,v_j)|$ of those, and set $r = \sum _{i,j} k_{ij}$.
\item Take the blow up of $\PP ^2$ at all the points chosen for each $i$ and $j$, which we will label by $p_1,...,p_r$:
\begin{equation}
X= \text{Bl} _{p_1, ... , p_r } \PP ^2,
\end{equation}
and consider the proper transform $\tilde{C}$ of $C$ in $X$.
\item Now, $X$ lives in the product $\PP ^2 \times \PP ^1$. Embed $X$ in $\PP ^{2+3-1} = \PP^4$ via a Segre embedding, and take the image of $\tilde{C}$. This will now be a projective curve with components of the correct genera (as the genus is a birational invariant), and any two components will intersect precisely at the correct number of points. Hence its dual graph will be $G$.

\end{enumerate}

\begin{example}
\label{ex12}
Consider the graph in Figure~\ref{Graph}. 
\begin{figure}[h]
\begin{center}
\includegraphics[height = 0.7 in]{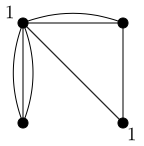}
\caption{The weighted graph in Example \ref{ex12}.}
\label{Graph}
\end{center}
\end{figure}

It has two components of genus zero and two components of genus one, which we can realize as a pair of lines (respectively, of cubics) in general position in $\PP^2$. The two lines will intersect in a point, the two cubics in nine points and each cubic will intersect each line in three points. The corresponding curve arrangement is as shown in Figure~\ref{Cubics}.
\begin{figure}[h]
\begin{center}
\includegraphics[height = 1.5 in]{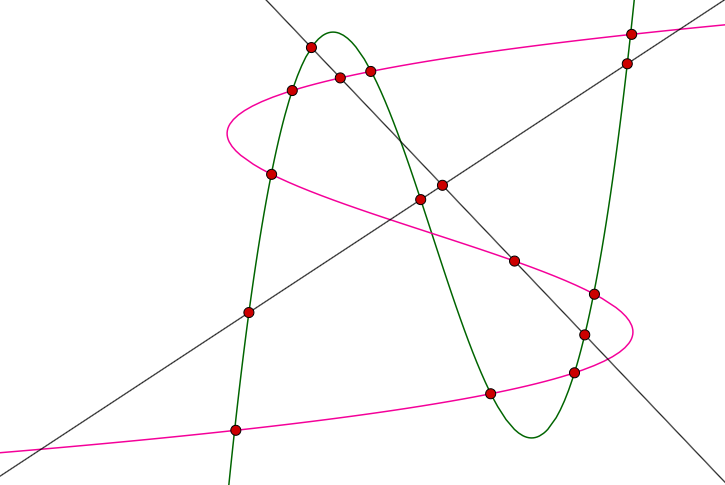}
\caption{Arrangement of curves in $\PP ^2$.}
\label{Cubics}
\end{center}
\end{figure}

We need to blow up eight of the nine intersection points between the two cubics, since they correspond to edges between the components of genus one in the graph. Moreover, the two components of genus zero do not share an edge, hence the unique intersection point between the two lines must be blown up, as well as the three intersection points of a chosen cubic with a line, two out of the three intersection point with the remaining line, and one of the three intersection points of the first cubic with the second line. In Figure~\ref{Cubics}, these point are marked in red. The result will be a curve in $\PP^2\times \PP^1\hookrightarrow \PP^4$, whose components will have the correct genera and will intersect at the correct number of points, and whose equations can be explicitly computed.

\end{example}

\begin{remark} The theory shows that it is, in principle, possible to find a smooth curve over $\Bbbk$ with a prescribed metric graph as its tropicalization. For certain types of graphs, more work has been done in this direction \cite{CFPU}. However, this problem is far from being solved in full generality in an algorithmic way, and this could be the subject of future research.
\end{remark}

\begin{acknowledgement}
This article was initiated during the Apprenticeship Weeks (22 August-2 September 2016), led by Bernd Sturmfels, as part of the Combinatorial Algebraic Geometry Semester at the Fields Institute. 
We heartily thank Bernd Sturmfels for leading the Apprenticeship Weeks
and for providing many valuable insights, ideas and comments. We would
also like to thank Melody Chan, for providing several useful
directions which were encoded in the paper and for suggesting
Example~\ref{melodyexample}. We are grateful to Renzo Cavalieri, for a
long, illuminating conversation about admissible covers, and Sam Payne
and Martin Ulirsch for suggesting references and for clarifying some
obscure points. We also thank Achill Sch\"urmann and Mathieu Dutour Sikiri\'c for their input on software for working with Delaunay subdivisions. 
The first author was supported by the Fields Institute for Research in Mathematical Sciences.
The second author was supported by the National Science Foundation Graduate Research Fellowship under Grant No. DGE 1106400 and the Max Planck Institute for Mathematics in the Sciences, Leipzig.
The third author was supported by a UC Berkeley University Fellowship and the Max Planck Institute for Mathematics in the Sciences, Leipzig.
\end{acknowledgement}

\Urlmuskip=0mu plus 1mu\relax
\bibliographystyle{amsalpha}
\bibliography{ref}

\end{document}